\documentclass[a4paper, 12pt]{article}

\usepackage[sort&compress]{natbib}
\bibpunct{(}{)}{;}{a}{}{,} 

\usepackage{float}
\usepackage{amsthm, amsmath, amssymb, mathrsfs, multirow, url}
\usepackage{graphicx} 
\usepackage{ifthen} 
\usepackage{amsfonts}
\usepackage[usenames]{color}
\usepackage{fullpage}
\usepackage{subfigure}
\usepackage{textcomp}
\usepackage{multirow}
\usepackage{dsfont}
\usepackage{tabularx}
\usepackage{bigints}

\theoremstyle{plain} 
\newtheorem{thm}{Theorem}
\newtheorem{cor}{Corollary}
\newtheorem{prop}{Proposition}
\newtheorem{lem}{Lemma}

\theoremstyle{definition}

\theoremstyle{remark}

\newtheorem{cond}{Condition}
\newtheorem*{rsthm}{Robbins--Siegmund Theorem}

\newcommand{\E}{\mathsf{E}}

\newcommand{\unif}{{\sf Unif}}
\newcommand{\nm}{{\sf N}}

\newcommand{\A}{\mathcal{A}} 
 
\newcommand{\RR}{\mathbb{R}}

\newcommand{\XX}{\mathbb{X}}

\newcommand{\UU}{\mathbb{U}}

\newcommand{\Phat}{\widehat{P}}
\renewcommand{\P}{\mathscr{P}}
\newcommand{\M}{\mathscr{M}}

\newcommand{\Ptilde}{\widetilde P}

\title{Revisiting consistency of a recursive estimator of mixing distributions}

\author{Vaidehi Dixit\footnote{Department of Statistics, North Carolina State University; {\tt vdixit@ncsu.edu}, {\tt rgmarti3@ncsu.edu}} \quad and \quad Ryan Martin$^*$}
\date{\today}

\begin{document}

\maketitle 

\begin{abstract}
Estimation of the mixing distribution under a general mixture model is a very difficult problem, especially when the mixing distribution is assumed to have a density. {\em Predictive recursion} (PR) is a fast, recursive algorithm for nonparametric estimation of a mixing distribution/density in general mixture models.  However, the existing PR consistency results make rather strong assumptions, some of which fail for a class of mixture models relevant for monotone density estimation, namely, scale mixtures of uniform kernels. In this paper, we develop new consistency results for PR under weaker conditions. Armed with this new theory, we prove that PR is consistent for the scale mixture of uniforms problem, and we show that the corresponding PR mixture density estimator has very good practical performance compared to several existing methods for monotone density estimation. 

\smallskip

\end{abstract}

\section{Introduction}

Mixture models are widely used in statistics and machine learning, often for density estimation and clustering.  Here we will be considering a general version of the mixture model, where the mixture density is given by 
\begin{equation}
\label{mixture}
    m_{P}(x) = \int_\UU k(x \mid u) \, P(du),
\end{equation}
where $k$ is a known kernel, i.e., where $x \mapsto k(x \mid u)$ is a density for each $u \in \UU$, and $P$ is the unknown mixing distribution on (the Borel $\sigma$-algebra of) $\UU$.  An advantage to this general form is its flexibility: depending on the kernel, the mixture density $m_P$ can take virtually any shape \citep[e.g.,][p.~572]{dasgupta2008}, making such mixtures a powerful modeling tool for robust, nonparametric density estimation.  Here we will assume that we have independent and identically distributed observations from a density $m$---which may or may not have the form \eqref{mixture}---and our goal is to fit the above mixture model, estimate the mixing distribution $P$, and, in turn, estimate the density $m$.  

An alternative perspective on the mixture model formulation considers a hierarchical formulation, where the first layer has iid $\UU$-valued random variables, $U_1,\ldots,U_n$, from $P$, and then the second layer has 
\[ (X_i \mid U_i) \sim k(x \mid U_i), \quad \text{independent, $i=1,\ldots,n$}. \]
The idea is that the $U_i$'s are latent/unobservable variables and the $X_i$'s are the observable data.  It is easy to check that, marginally, the $X_i$'s are iid with density $m_P$ as in \eqref{mixture}.  The classical deconvolution problem \citep[e.g.,][]{stefanski1990, fan1991} is a special case where $k(x \mid u)$ is such that the second layer above could be described as ``$X_i = U_i + \text{noise}$.''  This hierarchical formulation sheds light on the difficulties of the problem we are considering; that is, our goal is to estimate the distribution $P$ of the latent variables $U_1,\ldots,U_n$ based only on the corrupted observations $X_1,\ldots,X_n$. 

For fitting the general mixture model \eqref{mixture}, a number of different strategies are available in the literature. 
A natural approach is to use the nonparametric maximum likelihood estimator (MLE) of $P$ \citep{lindsay1995,eggermont1995} and the corresponding plug-in estimate of the mixture density $m_P$. An interesting feature of the nonparametric MLE of $P$ is that it is almost surely a discrete distribution \citep[e.g.,][]{lindsay1995}. Another approach is to assume discreteness of $P$ with a fixed number of components and the component parameters are estimated via EM \citep{dempster1977,mclachlanpeel2000, teel2015}. Bayesian approaches have also been explored in this context; either by having a prior on $P$ like in \citet{vandykmeng2001}, or a prior on the number of components of $P$ like in \citet{richardsongreen1997}.


An alternative to the likelihood-based frameworks mentioned above, \citet{newtonetal1998} proposed a recursive algorithm for nonparametric estimation of $P$, originally designed to serve as an approximation of the posterior mean under the Dirichlet process mixture formulation; see, also, \citet{newtonzhang1999}.  The so-called {\em predictive recursion} (PR) algorithm estimates the mixing distribution recursively, starting with an initial guess $P_0$ and applies a simple update $(P_{i-1}, X_i) \mapsto P_i$, for each $i=1,\ldots,n$, resulting in an estimate $P_n$ of $P$ and a corresponding estimate $m_n = m_{P_n}$ of $m$.  Advantages of the PR estimator include its speed and---compared to likelihood-based methods whose estimates of $P$ are effectively discrete---its ability to estimate a mixing distribution that has a smooth density with respect to any user-specified dominating measure.  Further details about the PR algorithm and its properties are discussed in Section~\ref{pr}.  

Not being likelihood-based makes the theoretical justification of the PR estimator of $P$ not straightforward.  It was not until \citet{newton2002} that a first theoretical convergence analysis of PR was presented, establishing the asymptotic consistency of the PR estimator $P_n$ as $n \to \infty$.  Unfortunately, there was a gap in Newton's proof, later filled by \citet{ghoshtokdar2006}.  These first results, along with those in \citet{martinghosh2008}, focus primarily on the case where $\UU$ is a known finite set.  \citet{tmg2009} extended the consistency results to the case of compact $\UU$, which was extended further by \citet{martintokdar2009} who covered the case of model misspecification, where the true density $m$ need not have exactly the form \eqref{mixture}, and bounded the rate of convergence. 

However, even the latter results are based on conditions that can be too restrictive in applications.  For example, \citet{williamson1955} showed that monotone densities are characterized as mixtures of the form \eqref{mixture} where $\UU = [0,\infty)$ and $k(x \mid u) = u^{-1} 1_{[0,u]}(x)$ is the uniform kernel, $\unif(x \mid 0, u)$, with $1_A(x)$ being the indicator function of a set $A$.  But for this particular kernel, it is not possible to check the sufficient conditions required in, e.g., Theorem~4.5 of \citet{martintokdar2009}.  Similar issues would arise in other mixture model applications.  Motivated by this deficiency in the state of the art, the focus of the present paper is to establish new asymptotic consistency properties for the PR estimator under weaker and more easily verified conditions.  

Following a brief review of the existing theory for PR in Section~\ref{pr}, we establish convergence properties of the PR estimator---both the mixture density and the mixing distribution---under weaker conditions in Section~\ref{results}. We then apply these new results in Section~\ref{monotone} to our motivating example, namely, monotone density estimation via mixtures of uniform kernels.  There we first give a characterization of the best mixing distribution and mixture density within a special class of uniform mixtures.  This characterization suggests a particular formulation of the PR algorithm and we use the general results presented in Section~\ref{results} to prove that PR consistently estimates this best mixture. Our choice to focus on a special class of uniform mixtures generally introduces some model misspecification bias, but we show that this bias is a vanishing function of two user-specified parameters. Therefore, the bias has no practical impact on PR's performance, as our numerical examples confirm. Finally, some concluding remarks are given in Section~\ref{S:discuss}.  Technical details and proofs are presented in the Appendix. 


\section{Background on PR}
\label{pr}

As mentioned briefly above, PR is a stochastic algorithm designed for fast, nonparametric estimation of mixing distributions.  The algorithm's inputs include the kernel $k$, an initial guess $P_0$ of the mixing distribution, supported on $\UU$, a rule for defining a sequence of weights $i \mapsto w_i \in (0,1)$, and a sequence of data points $X_1,X_2,\ldots$.  Then the recursive updates first presented in \cite{newtonetal1998}  define the PR algorithm:
\begin{equation}
\label{eq:pr}
P_i(du) = (1-w_i) \, P_{i-1}(du) + w_i \frac{k(X_i \mid u) \, P_{i-1}(du)}{\int_\UU k(X_i \mid v) \, P_{i-1}(dv)}, \quad u \in \UU, \quad i \geq 1.
\end{equation}
After $n$ data points have been observed, the mixing distribution estimator is $P_n$, and the corresponding mixture density estimator is $m_n = m_{P_n}$ defined according to \eqref{mixture}. To understand the motivation behind PR, observe that the $i^{th}$ PR update is just a weighted average of $P_{i-1}(du)$ and the posterior for $U$ with prior $P_{i-1}(du)$ and kernel likelihood $k(X_i \mid u)$. The weights, $w_i$, need to be decreasing in $i$ but not too quickly; this will be made more precise below.  Some recent and novel applications of PR can be found in \citet{scott.FDRreg}, \citet{scott.fdrsmooth}, and \citet{woody.scott.post}. 

If the user has a specific dominating measure $\mu$ on $\UU$ in mind, then he/she can incorporate that information into the algorithm.  That this, the updates in \eqref{eq:pr} can be expressed in terms of the density or Radon--Nikodym derivative $p_i = dP_i/d\mu$ as 
\[ p_i(u) = (1-w_i) \, p_{i-1}(u) + w_i \frac{k(X_i \mid u) \, p_{i-1}(u)}{\int_\UU k(X_i \mid v) \, p_{i-1}(v) \, \mu(dv)}, \quad u \in \UU, \quad i \geq 1, \]
where $p_0 = dP_0/d\mu$ is the initial guess.  Therefore, PR can be used to estimate a {\em mixing density}, compared to the nonparametric MLE, $\Phat$, which is almost surely discrete.  Moreover, when the densities are evaluated on a fixed grid in $\UU$, and the normalizing constant in the denominator is evaluated using quadrature, computation of the PR estimate, $P_n$, is fast and simple---done in $O(n)$ operations---compared to the nonparametric MLE or a Bayesian estimate based on Markov chain Monte Carlo (MCMC).

The above algorithm is described for the case when data points are arriving one at a time, but, of course, the same procedure can be carried out when the data $X_1,\ldots,X_n$ comes in a batch.  When data are both batched and iid, as we consider here, one might be troubled by the fact that $P_n$ depends on the order in which the data are processed.  In particular, while there are some potential advantages to PR's order-dependence (see \cite{dixitmartin2019}), it implies that $P_n$ is not a function of a minimal sufficient statistic.  To overcome this, \citet{newton2002} suggested that one could evaluate the estimator $P_n$ separately on several random permutations of the data sequence and then take averages over permutations.  This can be seen as a Monte Carlo estimate of the Rao--Blackwellized estimator, the average over all permutations.  It has been shown empirically \citep[e.g.,][]{mt-test} that it only takes a few random permutations to remove the order-dependence, so, with the inherent computational efficiency of PR, the permutation-averaged version is still much faster than, say, MCMC.  

Not being Bayesian or maximum likelihood estimators, it is not immediately obvious that the PR estimates, $P_n$ and $m_n$, would have any desirable statistical properties.  It has, however, been shown that, under certain conditions, both $P_n$ and $m_n$ are consistent estimators.  Before stating these sufficient conditions for consistency, we need to describe {\em what} the PR estimates are estimating in general.  

Suppose the true density of the iid data $X_1,\ldots,X_n$ is $m^\star$.  Of course, there is generally no way to know if $m^\star$ can be expressed as a mixture model of the form \eqref{mixture} for a particular kernel, $k$.  When the mixture model is incorrectly specified, there is no ``$P^\star$'' for the PR estimator $P_n$ to converge to, and we cannot expect $m_n$ to be a consistent estimator of $m^\star$.  Instead, there may be a mixture density, $m^\dagger(x) = \int k(x \mid u) \, P^\dagger(du)$, that is ``closest'' to $m^\star$, and that $P_n$ and $m_n$ would converge to $P^\dagger$ and $m^\dagger$, respectively.  Proximity here is measured in terms of the Kullback--Leibler divergence, 
\[ K(m^\star, m) = \int \log\{ m^\star(x) / m(x) \} \, m^\star(x) \, dx. \]
More precisely, let $\P$ denote (a possibly proper subset of) the collection of probability distributions $P$ on $\UU$, and define the corresponding set of mixtures of the form \eqref{mixture} for a given kernel $k$, 
\[ \M = \M(k, \P) = \{m_P: P \in \overline{\P}\}, \]
where $\overline{\P}$ is the closure of $\P$ with respect to the weak topology, i.e., $\P$ plus all possible limits of weakly convergent sequences in $\P$.  To avoid vaccuous cases, we will assume that $K(m^\star, m)$ is finite for at least one $m \in \M$. This is not a trivial assumption, however; see Section~\ref{monotone}.  In this case, the ``best approximation'' of $m^\star$ in $\M$ is the {\em Kullback--Leibler minimizer}, $m^\dagger$, that satisfies  
\begin{equation}
\label{eq:KL.minimizer}
K(m^\star, m^\dagger) = \inf\{K(m^\star, m): m \in \M\} 
\end{equation}
A relevant question is whether such a minimizer exists and if it is unique. Assuming that $K(m^\star,m)$ is finite for at least one $m \in \M$ and given that it is a convex function, we can expect that a minimizer $m^\dagger$ exists and is unique. Existence of a $P^\dagger$ corresponding to $m^\dagger$ is guaranteed by assuming certain conditions on $k$ and $\UU$; see Conditions~A1 and A2 in \citet{martintokdar2009} and, more generally, \citet[][Ch.~8]{liesevadja}. However, uniqueness of $P^\dagger$ requires identifiability of the mixture model \eqref{mixture} in $P$. 

In \citet{tmg2009}, consistency of the PR estimators was established in the case where the mixture model was correctly specified, i.e., when $m^\star \in \M$, so that there exists a true $P^\star \in \P$.  That is, under certain conditions, they showed $K(m^\star, m_n) \to 0$ almost surely and that $P_n \to P^\star$ weakly almost surely.  \citet{martintokdar2009} extended these consistency results to the case where the mixture model is not necessarily correctly specified, i.e., where possibly $m^\star \not\in \M$.  This extension is a practically important one, as it provides a theoretical basis for the PR-based marginal likelihood estimation framework developed in \citet{martintokdar2011} and later applied in, e.g., \citet{martinhan2016}, \citet{dixitmartin2020}.  Under conditions slightly stronger than those given in \cite{tmg2009} for the correctly specified case, they showed that $K(m^\star, m_n) \to K(m^\star, m^\dagger)$ and $P_n \to P^\dagger$ weakly, both almost surely. This implies, for example, that the PR estimates do the best they could, asymptotically, relative to the specified model. It turns out, however, that the sufficient conditions stated in \citet{martintokdar2009}, very similar to those in \citet{tmg2009}, are rather restrictive.  The most problematic of those assumptions is the following: 
\begin{equation}
\label{eq:old.bound}
\sup_{u_1, u_2 \in \UU} \int \Bigl\{ \frac{k(x \mid u_1)}{k(x \mid u_2)} \Bigr\}^{2} m^\star(x) \, dx < \infty.
\end{equation}
For nice kernels like $k(x \mid u) = \nm(x \mid u, \sigma^2)$ for a fixed $\sigma^2 > 0$, if $\UU$ is compact and $m^\star$ has Gaussian-like tails, then \eqref{eq:old.bound} can be satisfied.  However, if $m^\star$ is heavier-tailed, then \eqref{eq:old.bound} could easily fail.  More concerning is if we are considering a not-so-nice kernel, such as uniform: $k(x \mid u) = \unif(x \mid 0, u)$, for $x > 0$ and $u > 0$; this is the natural kernel in the case where $m^\star$ is monotone non-increasing on $[0,\infty)$.  In this case, the $u$-dependent support implies that the ratio in the above display is infinite on an open interval and, hence, \eqref{eq:old.bound} obviously fails.  The difficulty in verifying condition \eqref{eq:old.bound} in several practical applications is what motivated our investigation into potentially weaker sufficient conditions and, in turn, the present paper.

\section{New consistency results}
\label{results}

\subsection{Conditions}
\label{SS:setup}

The goal is to develop a new set of sufficient conditions for PR consistency that are weak enough that they can be checked in the applications we mentioned above, in particular, the case of uniform kernels for monotone density estimation.  First we make clear the setup/conditions, and then we present the main results.  

\begin{cond} 
\label{cond:weights}
The PR algorithm's weights satisfy $w_i = a(i+1)^{-1}$, for $a < \frac29$.
\end{cond}

\begin{cond}
\label{cond:support}
The mixing distribution support, $\UU$, is compact.
\end{cond}


\begin{cond}
\label{cond:integrable}
The kernel, the initial guess $P_0$, with corresponding $m_0 = m_{P_0}$, and the true $m^\star$ satisfy the following integrability property:
\begin{equation}
\label{eq:new.bound}
\sup_{u \in \UU} \int \Bigl\{ \frac{k(x \mid u)}{m_0(x)} \Bigr\}^2 \, m^\star(x) \, dx < \infty.
\end{equation}
\end{cond}

In the previous literature on this topic, and also in the literature on stochastic approximation more generally, the weights/step sizes are assumed to satisfy 
\[ w_i > 0, \quad \sum_{i=1}^\infty w_i = \infty, \quad \text{and} \quad \sum_{i=1}^\infty w_i^2 < \infty. \]
Of course, the specific weights in Condition~\ref{cond:weights}---which are of the same form as the weights used in \citet{hahn.martin.walker.pred}---satisfy these conditions, but others do to.  The reason we adopt this specific choice is that it allows us to replace \eqref{eq:old.bound} with the weaker bound \eqref{eq:new.bound} discussed more below.  And since the choice of weights is entirely in the hands of the user, while the choice of kernel may be determined by the context of the problem and $m^\star$ is a choice made by ``Nature'' and hidden from the user, it is best to sacrifice on generality in directions the user can control.  

Condition~\ref{cond:support} assumes that the mixing distribution support is compact, but this is not much of a restriction in practice, since it can be taken as large as the user pleases.  Compactness of $\UU$ is not strictly needed for the results presented below, but (a)~some more complicated notion of compactness is needed, as we briefly discuss in the paragraph leading up to Corollary~\ref{cor:mixing}, and (b)~Condition~\ref{cond:integrable} might be difficult to check without $\UU$ being compact.  For these reasons, we opt for the simpler albeit slightly more restrictive compactness condition listed above.  

Finally, the most complicated assumption is in Condition~4, about integrability.  To understand this better, it may help to re-express the integrand as 
\[ \frac{k(x \mid u)}{m_0(x)} \cdot \frac{m^\star(x)}{m_0(x)} \cdot k(x \mid u). \]
First, if the PR prior guess $P_0$ is not too tightly concentrated, then the mixture $m_0$ would be heavier-tailed than any individual kernel $k(\cdot \mid u)$.  In that case, the first ratio in the above display would be bounded, or at least would not increasing too rapidly.  Second, we cannot expect PR, or any mixture model-based method for that matter, to be able to do a good job of estimating $m^\star$ if a mixture with a relatively diffuse mixing distribution cannot adequately cover the support of $m^\star$.  So the heart of Condition~\ref{cond:integrable} is an assumption that the posited mixture model {\em can} adequately cover the support of $m^\star$, in the sense that the second ratio in the above display is not blowing up too rapidly.  Finally, if the two ratios are well controlled, then the integral with respect to $k(\cdot \mid u)$ should be bounded uniformly in $u$.  We shall see below, in Section~\ref{monotone}, that \eqref{eq:new.bound} can be checked for uniform kernels while the condition \eqref{eq:old.bound} in \citet{martintokdar2009} cannot. 




\subsection{Main results}
\label{consistency}

Our goal in this section is to show that the PR estimator, $m_n = m_{P_n}$, of $m^\star$ is consistent in the sense that $K(m^\star, m_n)$ converges almost surely to $\inf_{m \in \M} K(m^\star, m)$, the minimum Kullback--Leibler divergence over the posited mixture model class $\M$.  In the special case where $m^\star \in \M$, this implies consistency in the usual sense: $K(m^\star, m_n) \to 0$ almost surely.  In either case, it says that the PR estimator, $m_n$, is close to the best possible mixture approximation of $m^\star$, at least asymptotically.  We will also show how consistency of the mixing distribution estimator can be established from consistency of the mixture, but this will require further explanation; see below.  

\begin{thm}
\label{thm:mixture.KL}
Under Conditions~\ref{cond:weights}--\ref{cond:integrable}, the PR estimator, $m_n$, of the density $m^\star$ satisfies $K(m^\star, m_n) \to \inf_{m \in \M} K(m^\star, m)$ almost surely.  In particular, if $m^\star \in \M$, then $K(m^\star, m_n) \to 0$ almost surely.
\end{thm}

\begin{proof}
See Appendix~\ref{SS:proof.thm1}.
\end{proof}

Here we give a very rough sketch of the proof strategy.  Start by writing $K_n = K(m^\star, m_n) - \inf_{m \in \M} K(m^\star, m)$, and let $\A_i$ denote the $\sigma$-algebra generated by the observations $X_1,\ldots,X_i$, for $i=1,2,\ldots$.  We show in the proof that 
\[ \E(K_n \mid \A_{n-1}) = K_{n-1} - w_n T(P_{n-1}) + w_n^2 \E(Z_n \mid \A_{n-1}), \quad n \geq 1, \]
where 
\begin{equation}
\label{eq:T}
T(P) = \int_\UU \Bigl\{ \int_\XX \frac{m(x)}{m_P(x)} \, k(x \mid u) \, dx \Bigr\}^2 \, P(du) - 1, 
\end{equation}
and $Z_n$ is a ``remainder'' term defined in the appendix.  It follows from Jensen's inequality that $T(P) \geq 0$, with equality if and only if $P=P^\dagger$, the Kullback--Leibler minimizer.  If we could ignore the remainder term, then $K_n$ would be a non-negative supermartingale and, therefore, would converge almost surely to some $K_\infty$.  Of course, the remainder term cannot be ignored, so we will use the ``almost supermartingale'' results in \citet{robbinssiegmund} to accommodate this.  Moreover, to show that $K_\infty$ is 0 almost surely, we will use some new and useful properties of the function $T$ in \eqref{eq:T} which were overlooked in the analysis presented in \citet{martintokdar2009}.

When the mixture model is correctly specified, so that $m^\dagger=m^\star$, it follows from Theorem~\ref{thm:mixture.KL} and the familiar properties of Kullback--Leibler divergence that $m_n \to m^\star$ almost surely in Hellinger or total variation distance, i.e., that $\int (m_n^{1/2} - m^{\star 1/2})^2 \, dx$ and $\int |m_n - m^\star| \, dx$ both go to 0 almost surely.  In the general case where the mixture model is misspecified, Theorem~\ref{thm:mixture.KL} still strongly suggests that $m_n \to m^\dagger$, but some effort is required to connect the Kullback--Leibler difference to a distance between $m_n$ and $m^\dagger$.  Towards this, define the {\em Hellinger contrast} $\rho(m_1,m_2) = \rho_{m^\star}(m_1,m_2)$, which is given by 
\[ \rho^2(m_1, m_2) = \int (m_1^{1/2} - m_2^{1/2})^2 (m^\star / m^\dagger) \, dx. \]
This is just a weighted version of the ordinary Hellinger distance---with weight function $m^\star / m^\dagger$---so it is a proper metric. Clearly, if the mixture model is correctly specified, so that $m^\dagger = m^\star$, then $\rho$ is exactly the Hellinger distance.  See \cite{patilea2001} and \cite{kleijnvaart2006} for further details on the Hellinger contrast.  The following result establishes that $\rho(m^\dagger, m_n) \to 0$ almost surely, which implies that the limit $m_\infty$ of $m_n$ satisfies $m_\infty = m^\star$ almost everywhere with respect to the measure with Lebesgue density $m^\star$.  Under some additional conditions, namely, that $m^\dagger$ is suitably close to $m^\star$, the PR estimator $m_n$ is shown to converge to $m^\dagger$ in total variation distance, which implies the limit is equal to $m^\dagger$ almost everywhere with respect to Lebesgue measure. 

\begin{cor}
\label{cor:mixture.L1}
Under the conditions of Theorem~\ref{thm:mixture.KL}, $\rho(m^\dagger, m_n) \to 0$ almost surely.  Moreover, if $m^\dagger/m^\star \in L_\infty(m^\star)$, then $m_n \to m^\dagger$ almost surely in total variation. 
\end{cor}

\begin{proof}
See the proof of Corollary~4.10 in \citet{martintokdar2009}.
\end{proof}

Finally, what can be said about the convergence of the mixing distribution estimator, $P_n$?  Again, Theorem~\ref{thm:mixture.KL} strongly suggests that $P_n$ is converging to $P^\dagger$ in some sense, but we cannot make that leap immediately.  In particular, without additional assumptions, there is no guarantee that $P^\dagger$ is unique or even that $P_n$ converges at all.  For this, we will need {\em identifiability} of the mixture model \eqref{mixture} and tightness of $(P_n)$. Under Condition~2, as we assume here, tightness of $P_n$ follows from Prokhorov's theorem.  If compactness of $\UU$ is not a feasible assumption, then one can instead verify the more general sufficient condition, namely, Condition~A6 in \citet{martintokdar2009}, for tightness of $P_n$. 

We will also reqjuire the following fairly abstract condition on the kernel density $k$, written in terms of a general sequence of mixing distributions $(Q_t)$ on $\UU$:
\begin{equation}
\label{eq:kernel}
\text{$Q_t \to Q_\infty$ weakly implies $m_{Q_t}(x) \to m_{Q_\infty}(x)$ for almost all $x$}. 
\end{equation}
In words, \eqref{eq:kernel} states that the kernel is such that weak convergence of mixing distributions implies almost everywhere pointwise convergence of mixture densities.  This holds immediately if $u \mapsto k(x \mid u)$ is bounded and continuous for almost all $x$, as was assumed in \citet{martintokdar2009} and elsewhere.  However, in some examples, like in Section~\ref{monotone} below, strict continuity of the kernel fails, but condition \eqref{eq:kernel} can be verified.  

\begin{cor}
\label{cor:mixing}
In addition to the conditions of Theorem~\ref{thm:mixture.KL}, assume that
\begin{itemize}
\item the model is identifiable, i.e., $m_P = m_{P'}$ almost everywhere implies $P = P'$,
\item the kernel is such that \eqref{eq:kernel} holds, 
\item and $m^\dagger/m^\star \in L_\infty(m^\star)$. 
\end{itemize}
Then the Kullback--Leibler minimizer $P^\dagger$ is unique and $P_n \to P^\dagger$ weakly almost surely. 
\end{cor}

\begin{proof}
Since $P_n$ is tight, there exists a subsequence $P_{n(t)}$ such that $P_{n(t)} \to P_\infty$ weakly, for some $P_\infty$.  By \eqref{eq:kernel}, we have pointwise convergence of the mixture densities, i.e., $m_{n(t)}(x) \to m_\infty(x)$ for almost all $x$, and then $m_{n(t)} \to m_\infty$ in total variation distance thanks to Scheff\'e's theorem.  But Corollary~\ref{cor:mixture.L1} already gives us $m_n \to m^\dagger$ almost surely in total variation distance on the full/original sequence.  Therefore, it must be that $m_\infty = m^\dagger$ almost surely and, by identifiability, that $P_\infty = P^\dagger$.  Since any such convergent subsequence of $P_n$ would have the same almost weak limit, $P^\dagger$, it must be that $P_n$ itself converges weakly almost surely to $P^\dagger$, as claimed. 
\end{proof}

The boundedness assumption on $m^\dagger/m^\star$, as in Corollary~\ref{cor:mixture.L1}, is needed simply to convert convergence of $m_n$ to $m^\dagger$ in the Hellinger contrast to convergence in total variation.  Identifiability of the mixture model $m_P$ in $P$ is non-trivial. Additively-closed one-parameter families of distributions were proved to be identifiable in \citet{teicher1961}. Identifiability of finite mixtures of gamma and of Gaussian distributions was proved in \citet{teicher1963}. Scale mixtures of uniform distributions, like we discuss in Section~\ref{monotone} below, were shown to be identifiable in \citet{williamson1955}. More generally, identifiability of mixture models needs to be checked on a case-by-case basis.

\section{Application: Monotone density estimation}
\label{monotone}

\subsection{Background}

Any monotone non-increasing density can be written as a scale mixture of uniforms \citep{williamson1955}, i.e., for any monotone density $m$ defined on $\XX = [0, \infty)$, there exists a mixing distribution $P$, supported on $\UU = [0,\infty)$, such that,
\begin{equation}
\label{unif_mix}
   m(x) = \int_0^\infty \unif(x \mid 0, u) \, P(du),
\end{equation}
where $\unif(x \mid 0, u) = u^{-1} 1_{[0,u]}(x)$ is the uniform kernel density.  Therefore, the problem of estimating a monotone density can, at least in principle, be solved through the use of mixture density estimation methods, such as the PR algorithm.  

Let $X_1, \ldots, X_n$ be iid from a monotone non-increasing density $m^\star$. One approach to estimating $m^\star$ is to calculate the nonparametric MLE, also known as the {\em Grenander estimator} \citep{grenander1956}, which is the left derivative of the least concave majorant of the empirical distribution function. It is known that Grenander's is a consistent estimator of $m^\star$, with consistency results obtained in \citet{rao1969} and \citet{groeneboom1985}.  However, as shown in, e.g., \citet{woodroofesun1993}, the Grenander estimator tends to over-estimate near the origin and, in particular, is inconsistent at the origin. The same authors proposed a penalized likelihood estimator that penalizes the Grenander estimator at the origin and is also consistent overall.  

Another approach is Bayesian, whereby a prior distribution on $m$ is imposed by using the mixture characterization in \eqref{unif_mix} along with a suitable prior on the mixing distribution $P$.  A natural choice is a Dirichlet process prior on $P$, leading to a Dirichlet process mixture of uniforms model for the density $m$; see \citet{bornkamp2009}. Although this approach seems straightforward, obtaining asymptotic consistency results for the posterior distribution is made difficult by the uniform kernel's varying support.  In particular, if the support for the mixing distribution is not suitably chosen, then the Kullback--Leibler divergence of a posited mixture model from the true density would be infinite, which creates problems for verifying the so-called ``Kullback--Leibler property'' \citep{schwartz1965, wu.ghosal.2008} in the classical Bayesian consistency theory.  Some strategies have been suggested in, e.g., \citet{salomond2014}, who showed that the Bayesian posterior distribution under the Dirichlet process mixture prior has a near optimal concentration rate in total variation.  More recently, \citet{martin2019} proposed the use of an empirical, or data-driven prior for which the prior support conditions required for asymptotic consistency are automatically satisfied, and showed that the corresponding empirical Bayes posterior distribution concentrates around the true monotone density at nearly optimal minimax rate.  But the fully Bayesian solutions are computationally non-trivial and somewhat time consuming; moreover, the estimates tend to be relatively rough.  The PR algorithm, which is computationally fast and tends to produce smooth estimates, is a natural alternative to the aforementioned likelihood-based methods.  

\subsection{PR for uniform mixtures}


Suppose that the true density $m^\star$ is any monotone density supported on $[0, \infty)$. We know that $m^\star$ can be written as a mixture in \eqref{unif_mix}, so there exists a mixing distribution $P^\star$, which is also supported on $[0, \infty)$.  This point is relevant because of the following unique feature of uniform mixtures: if $m_P$ is a mixture model as in \eqref{unif_mix} with $P$ supported on $[0,L)$, then $m_P(x) = 0$ for all $x > L$ and, hence, if $L < \infty$, then $K(m^\star, m_P) \equiv \infty$.  Therefore, the upper bound of $\UU$ being $\infty$ creates some serious challenges. 
For practical implementation of the PR algorithm, and for the theory as discussed above, a compact mixing distribution support is needed.  This calls for a different approach.  

For a fixed $L \in (0,\infty)$, define a new target, $m^{\star L}$, which is simply $m^\star$ restricted and renormalized to $[0,L)$.  That is, if $M^\star$ denotes the distribution function corresponding to the density $m^\star$, then 
\[ m^{\star L}(x) = \frac{m^\star(x) \, 1_{[0,L]}(x)}{M^\star(L)}. \]
Alternatively, $m^{\star L}$ can be viewed as the conditional density of $X$, given $X \leq L$; see below. The point of this adjustment is that $m^{\star L}$ has a known and bounded support, so a mixture model with mixing distribution supported on (a large subset of) $[0,L)$ can be fit with the PR algorithm to efficiently and accurately estimate this new target $m^{\star L}$.  Note that $m^{\star L}$ can be made arbitrarily close to $m^\star$ by choosing $L$ sufficiently large (see below), so this modification has no practical consequences. 








For technical and practical reasons, we cannot use the PR algorithm when the support of the mixing distribution contains $u=0$, so we introduce a new lower bound $\ell \in (0,L)$, which can be arbitrarily small.  Then the proposed mixture model to be fit by PR is 
\begin{equation}
\label{eq:monotone}
m_P(x) = \int_\UU \unif(x \mid 0, u) \, P(du), \quad x \in [0,L], \quad \UU = [\ell, L].  
\end{equation}
While both $m_P$ above and the adjusted target $m^{\star L}$ are supported on $[0,L]$, the model in \eqref{eq:monotone} is still {\em slightly misspecified} through the introduction of the lower bound $\ell > 0$ of the mixing distribution support.  In particular, note that $m_P(x)$ is constant for $x \in [0,\ell]$.  But the fact that $\ell$ can be taken arbitrarily small means that there are no practical consequences to this misspecification.  It does complicate the convergence analysis, but, fortunately, the theory presented in Section~\ref{results} above is general enough to handle this. 

Given that the mixture model \eqref{eq:monotone} is slightly misspecified, it is important to know what we can expect the PR algorithm to do. Theorem~\ref{thm:mixture.KL} states that, roughly, the PR estimator $m_n$ will converge to the Kullback--Leibler minimizer $m^\dagger$.  Since the supports of $m^{\star L}$ and the model densities $m_P$ in \eqref{eq:monotone} are the same, we avoid the ``$K(m^{\star L}, m_P) \equiv \infty$'' problem so minimizing the Kullback--Leibler divergence is well-defined.  To understand the bias coming from model misspecification, it will be important to understand what $m^\dagger$ looks like.  Incidentally, \citet{williamson1955} established that uniform mixtures are identifiable, so there is a unique mixing distribution, $P^\dagger$, supported on $\UU$, at which the Kullback--Leibler divergence is attained.  The following lemma gives the details.

\begin{lem}
\label{lem:KL.min1}
For the targeted monotone density $m^{\star L}$ supported on $[0,L]$, if the proposed mixture model is as in \eqref{eq:monotone}, then the unique minimizer, $P^\dagger = P^{\dagger \ell,L}$, of the Kullback--Leibler divergence $P \mapsto K(m^{\star L}, m_P)$ is given by 
\begin{equation}
\label{eq:mono.P.dagger}
P^\dagger = a_\ell \, \delta_{\{\ell\}} + a_\UU \, P^\star|_\UU + a_L \, \delta_{\{L\}}, 
\end{equation}
where $\delta_{\{t\}}$ is the Dirac point-mass at $t$, $P^\star|_\UU$ is $P^\star$ restricted to $\UU = [\ell,L]$, and the coefficients are given by 
\[ a_\ell = \frac{P^\star([0,\ell])}{M^\star(L)}, \quad a_\UU = \frac{P^\star([0,L])}{M^\star(L)}, \quad a_L = \frac{L m^\star(L)}{M^\star(L)}, \] 
with $M^\star$ the distribution function corresponding to $m^\star$. 
Then the best approximation of $m^{\star L}$ under model \eqref{eq:monotone} is $m^\dagger = m_{P^\dagger}$, given by 
\begin{equation}
\label{eq:mono.m.dagger}
m^\dagger(x) = a_\ell \, \unif(x \mid 0, \ell) + a_\UU \int_\UU \unif(x \mid 0, u) \, P^\star(du) + a_L \, \unif(x \mid 0, L). 
\end{equation}
\end{lem}

\begin{proof}
See Appendix~\ref{SS:proof.lemma1}. 
\end{proof}


The characterization result in Lemma~\ref{lem:KL.min1} is intuitive.  There is a true $P^\star$ that characterizes the true monotone mixture density $m^\star$, both generally supported on $[0,\infty)$.  Our proposed model, however, effectively restricts the mixing distribution's support to $[\ell,L]$, so it makes sense that the best approximation would agree with $P^\star$ on $[\ell,L]$ and then suitably allocate the remaining mass to the endpoints $\ell$ and $L$. 

From Section~\ref{pr}, recall that the implementation of the PR algorithm begins with an initial guess $P_0$, and that this effectively determines the dominating measure with respect to which $P_n$ has a density. PR's ability to choose the underlying dominating measure comes in handy in cases like this where we know that the target mixing distribution, $P^\dagger$, has an ``unusual'' dominating measure.  From Lemma \ref{lem:KL.min1}, we know that the best mixing distribution for fitting mixture model \eqref{eq:monotone} to $m^{\star L}$ puts point masses at the endpoints, $\ell$ and $L$, of $\UU$, and has a density with respect to Lebesgue measure on the interior of $\UU$. So, naturally, we can initialize the PR algorithm with a starting guess $P_0$ that has a density with respect to the dominating measure $\delta_{\{\ell\}} + \lambda_{\UU} + \delta_{\{L\}}$, where $\lambda_\UU$ denotes Lebesgue measure on $\UU$. Specifically, our proposal is to initialize the PR algorithm at 
\[ P_0 = p_{0,\ell} \, \delta_{\{\ell\}} + (1 - p_{0,\ell} - p_{0,L}) \, P_{0,\UU} + p_{0,L} \, \delta_{\{L\}}, \]
where $p_{0,\ell}$ and $p_{0,L}$ are positive with sum strictly less than 1, and $P_{0,\UU}$ has a density with respect to Lebesgue measure, e.g., $P_{0,\UU}$ could just be a uniform distribution on $\UU$.  Then the estimate, $P_n$, after the $n^\text{th}$ iteration will have the same form
\[ P_n = p_{n,\ell} \, \delta_{\{\ell\}} + (1 - p_{n,\ell} - p_{n,L}) \, P_{n,\UU} + p_{n,L} \, \delta_{\{L\}}, \]
and the corresponding mixture density estimate, $m_n$, is obtained as usual by integrating the kernel with respect to the mixing distribution $P_n$.


\subsection{Theoretical results}

Now that we know what the PR algorithm ought to converge to, we are ready to state our main result of this section.  First, a word about the notation/terminology that follows.  In our previous results, when we wrote ``almost surely,'' this was referring to the law that corresponds to iid sampling from $m^\star$.  In the results below, $m^{\star L}$ is the target, so we will write ``$m^{\star L}$-almost surely'' to be clear that it is with respect to the law corresponding to iid sampling from $m^{\star L}$.  Recall that $m^{\star L}$ is the conditional density of $X$, given $X \leq L$, so this modified law can be interpreted as iid sampling from $m^\star$, but throwing away any data points that exceed $L$.  Again, since $L$ can be taken arbitrarily large, there are no practical consequences of this restriction.  In fact, a bound on the bias induced by both the $L$- and $\ell$-restrictions is given in Proposition~\ref{prop:bias} below. 

\begin{thm}
\label{thm:monotone}
If $m^\star$ is renormalized to $m^{\star L}$ supported on $[0,L]$, and if the proposed mixture model $m_{P}$ is as in \eqref{eq:monotone}, then the PR estimator $m_n$ satisfies 
\[ K(m^{\star L}, m_n) \to K(m^{\star L}, m^\dagger), \quad \text{$m^{\star L}$-almost surely} \]
where $m^\dagger$ is as given in Lemma~\ref{lem:KL.min1}. Moreover, $m_n$ converges $m^{\star L}$-almost surely to $m^\dagger$ in total variation distance and the mixing distribution estimates $P_n$ converges weakly $m^{\star L}$-almost surely to $P^\dagger$ in \eqref{eq:mono.P.dagger}. 
\end{thm}

\begin{proof}
See Appendix~\ref{SS:proof.thm2}. 
\end{proof}

Our choice to restrict the mixing distribution's support to $\UU = [\ell, L]$ introduces some bias.  That is, the limit $m^{\dagger}$ of the sequence of PR estimators is the Kullback--Leibler minimizer over all mixtures supported on $\UU = [\ell, L]$, but it is different from $m^{\star L}$, which is different from $m^\star$. Intuitively, if $\ell \approx 0$ and $L \approx \infty$, then the bias ought to be negligible.  The next result confirms this intuition by bounding the bias as a function of $(\ell, L)$.  

\begin{prop}
\label{prop:bias}
The $L_1$ distance between the true monotone density $m^\star$ and the best approximation $m^\dagger$ in \eqref{eq:mono.m.dagger} under the restricted model \eqref{eq:monotone} is bounded as 
\begin{equation}
\label{eq:bias}
\int |m^{\dagger}(x) - m^\star(x)| \,dx \leq 2 \bigl\{ 1 - M^\star(L) + M^\star(L)^{-1} P^\star([0,\ell]) \bigr\}. 
\end{equation}
\end{prop}

\begin{proof}
See Appendix~\ref{proof:bias}.
\end{proof}

To make the bound in \eqref{eq:bias} more concrete, we consider a specific case.  A common choice in the literature \citep[e.g.,][]{salomond2014, martin2019} is to assume $m^\star$ has tails that vanish exponentially fast, so that $m^\star(x) \leq \exp(-bx^r)$, for all large $x$ and some positive constants $b$ and $r$; the case $r=\infty$ corresponds to $m^\star$ having a bounded support.  From this, and standard asymptotic bounds on the incomplete gamma function, it follows that $1-M^\star(L) \lesssim L^{-r} \exp(-b L^r)$, for large $L$. Furthermore, if, e.g., $P^\star$ has a bounded density at 0, then we have $P^\star([0,\ell]) \lesssim \ell$. Combining these two, we arrive at the following, more explicit bound on the $L_1$ bias as a function of $(\ell,L)$:
\[ \int |m^\star(x) - m^\dagger(x)| \, dx \lesssim L^{-r} e^{-b L^r} + \ell. \]
Clearly, by taking $\ell$ small and $L$ even just moderately large, the overall bias as a result of restricting to $\UU = [\ell, L]$ can be made negligibly small.

As a final technical detail in this section, we consider the problem of estimating $m^\star(0)$, the density at its mode, the origin.  This is an interesting and challenging problem, with a variety of applications; see, e.g., \citet{vardi1989}.  In particular, \citet{woodroofesun1993} highlight examples such as time between breakdowns of a system and distribution of galaxies that require the estimation of this modal $m^\star(0)$. The PR algorithm gives an obvious estimator of $m^\star(0)$, in particular, $m_n(0)$. The following result gives a theoretical basis for using this estimate and simulations in Section~\ref{simulation} show that the proposed estimate at 0 performs well when compared to existing methods.

\begin{prop}
\label{prop:origin}
Under the assumptions of Theorem~\ref{thm:monotone}, $m_n(0) \to m^\dagger(0)$ $m^{\star L}$-almost surely. Furthermore, the bias between $m^\dagger(0)$ and $m^\star(0)$ is bounded ,i.e,
\[ m^\dagger (0) - m^\star (0) \lesssim 1 - M^\star(L) \to 0, \quad \text{as $L \to \infty$}. \]
\end{prop}

\begin{proof}
See Appendix~\ref{proof:origin}. 
\end{proof}

\subsection{Numerical illustrations}
\label{simulation}

In this section we compare different methods for monotone density estimation to our PR-based method. The four methods we consider are the Grenander estimate, a Bayesian approach using a Dirichlet process, Bayesian approach using an empirical prior, and the method based on optimization of the penalized likelihood. The Grenander estimate is based on the nonparametric MLE and can be calculated easily using the R package {\tt fdrtool} \citep{fdrtool}. Settings for the Dirichlet process mixture and the empirical Bayes were based on those suggested in \citet{martin2019} and computed using the R codes he provided on his website.\footnote{\url{https://www4.stat.ncsu.edu/~rmartin/}} The penalized likelihood maximization was based on \citet{woodroofesun1993} and we used one of the values  recommended by those authors for their penalization parameter, i.e., $\alpha = n^{-1} \log n$. For PR, we take the mixing distribution support to be $\UU = [\ell,L]$, with $\ell = 10^{-5}$ and $L = \max(X)$. The initial guess $P_0$ is taken to be uniform on $\UU$. To reduce the dependence of the PR estimator on the data order, we average the PR estimates over 25 random permutations of the data.  For the comparisons below, we consider both real and simulated data sets.

First, we consider data coming from a study of suicide risks reported in \citet{silverman2018}, which consists of lengths of psychiatric treatment for $n = 86$ patients used as control. As per the detailed study of suicide risks in \citet{copasfryer1980}, there is a higher risk for suicide in the early stages of treatment, so modeling these data with a monotone density is appropriate. Figure~\ref{fig:hist} shows a comparison of the four monotone density estimation methods discussed above with PR over a histogram of the data. PR gives a smooth estimate of the monotone density in a very short amount of time, much faster than the Bayes and empirical Bayes estimates that require Markov chain Monte Carlo.  The take-away message is that, PR's misspecification bias---due to the choice of $\ell$ and $L$---can be easily controlled and that it gives a quality estimate compared to the other four methods.  In fact, the PR estimate in this case is smoother than the other four methods, a desirable feature in applied data analysis.  The simulations below will give a clearer picture of how PR performs compared to the other four methods. 



\begin{figure}[t]
\centering
\includegraphics[width = 0.6\linewidth]{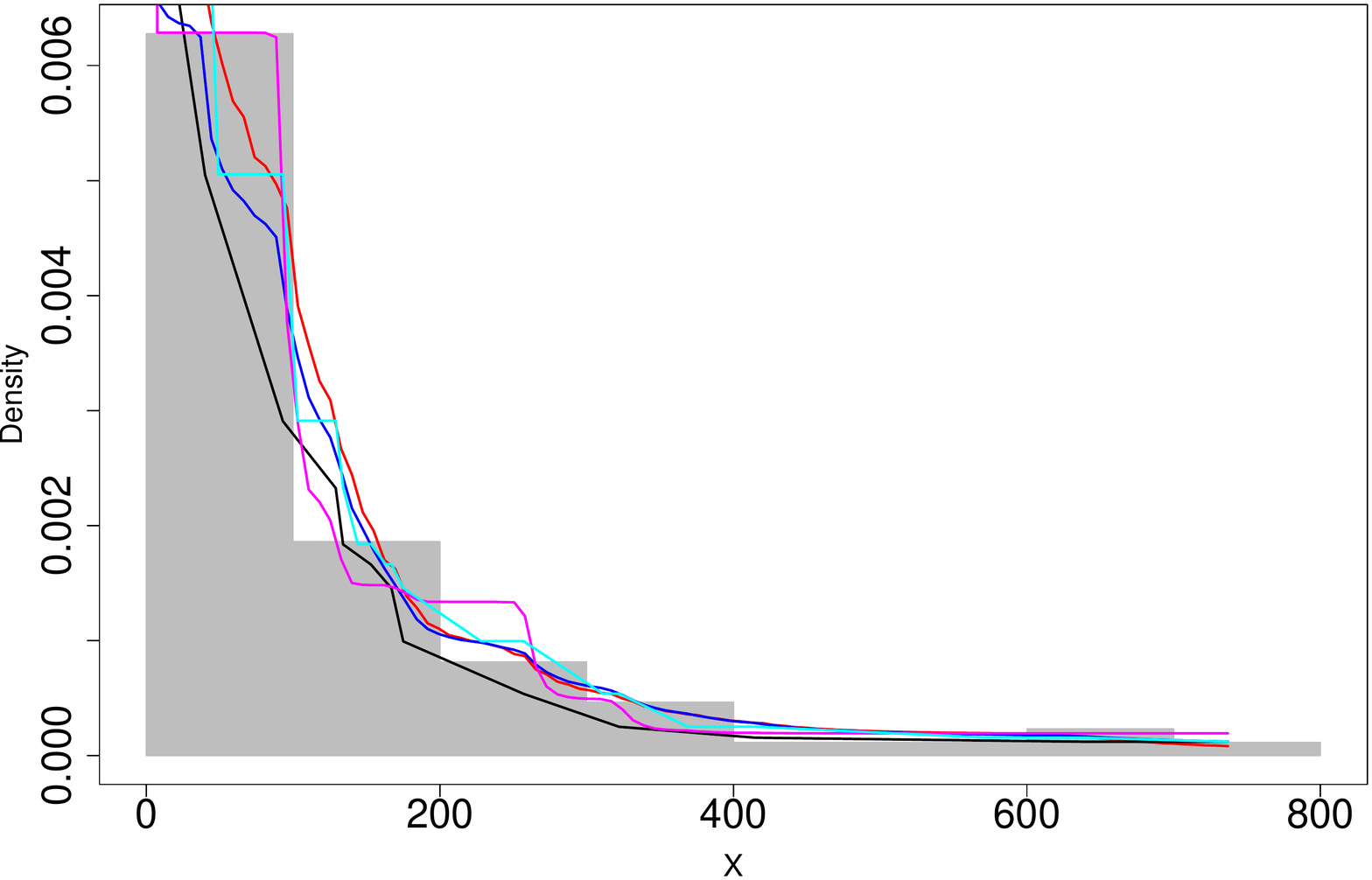}
\caption{A monotone density is fit to the suicide risk data from \citet{silverman2018} with the four different methods: PR (red), Grenander (black), empirical Bayes (blue), Bayes (magenta), and penalized likelihood (cyan).}
\label{fig:hist}
\end{figure}



Second, we consider two true monotone densities $m^\star$, namely, the standard exponential and the half standard normal. We carry out the simulation study over sample sizes of $n = 50, 100, 200$. For each $n$, we generate 200 data sets of size $n$ and produce the five different estimates on each data set.   As our metric of comparison, we use the total variation (or $L_1$) distance between the true density and the estimate. Additionally since inconsistency of the Grenander estimate at the origin is a well-known complication we also look at the ratio $\hat m(0)/m^\star(0)$ for each method. Boxplots summarizing both the $L_1$ distance and the at-the-origin ratio for the two simulations are shown in Figures~\ref{fig:boxplot_exp} and \ref{fig:boxplot_norm}. Consider the boxplots summarizing the $L_1$ distance. As the sample size increases, the boxplots for all five methods shrink towards 0, as expected. Notably, performance of PR is better than the Grenander estimator over all sample sizes. It is also faster and with slightly better performance when compared to the two Bayesian estimates, and is comparable to the penalized likelihood estimate. For estimating the density at 0, we compare PR with only the state-of-art estimates, namely the one based on penalizing the nonparametric MLE near 0 and the DP mixture. Even though PR is not tailored specifically for estimation at 0, as the penalized likelihood estimator is, its performance is competitive with the other methods.

\begin{figure}[t]
\centering
\subfigure[$L_1$ distance, $n = 50$]
{\includegraphics[width = 0.32\linewidth]{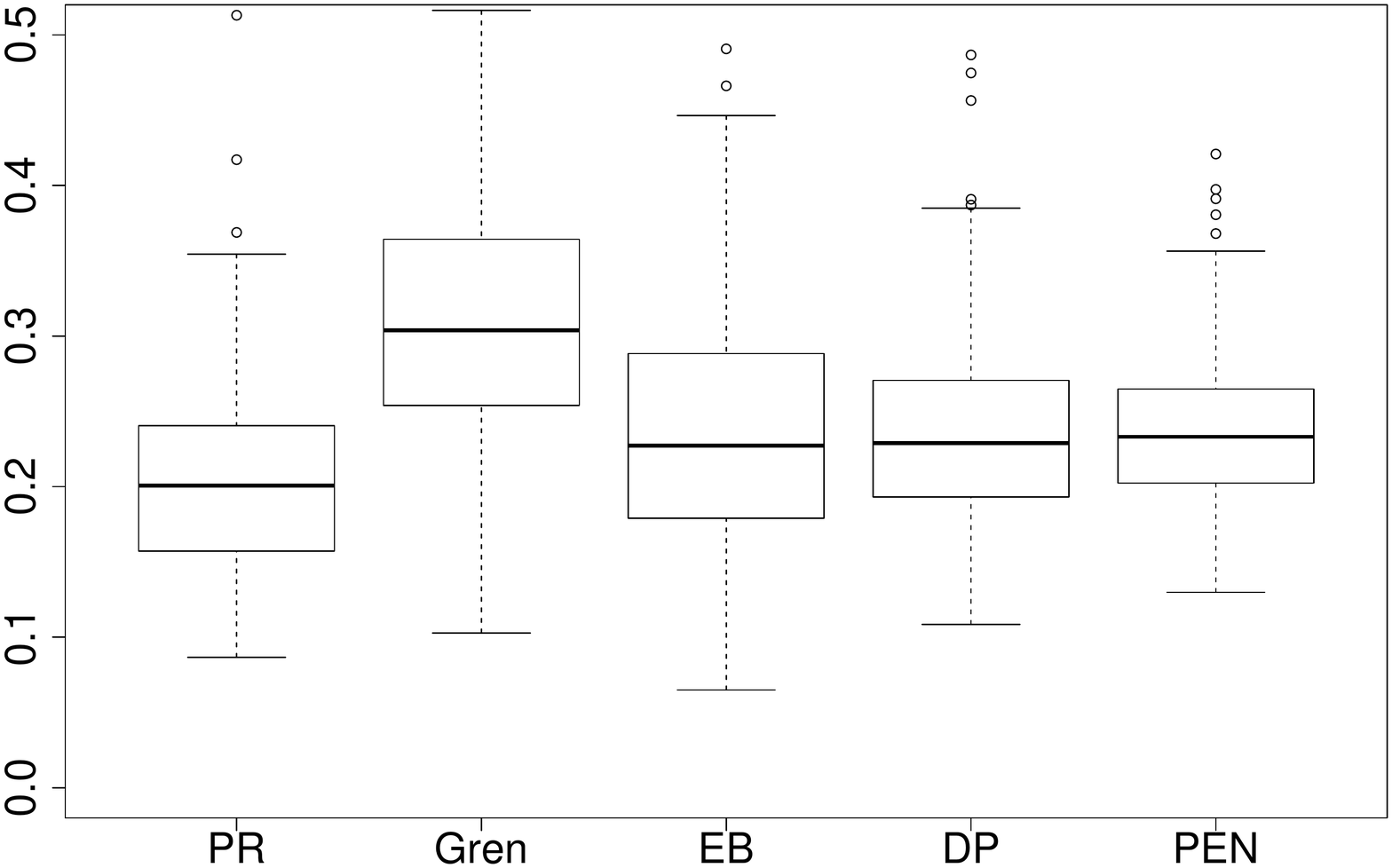}} 
\subfigure[$L_1$ distance, $n = 100$]
{\includegraphics[width = 0.32\linewidth]{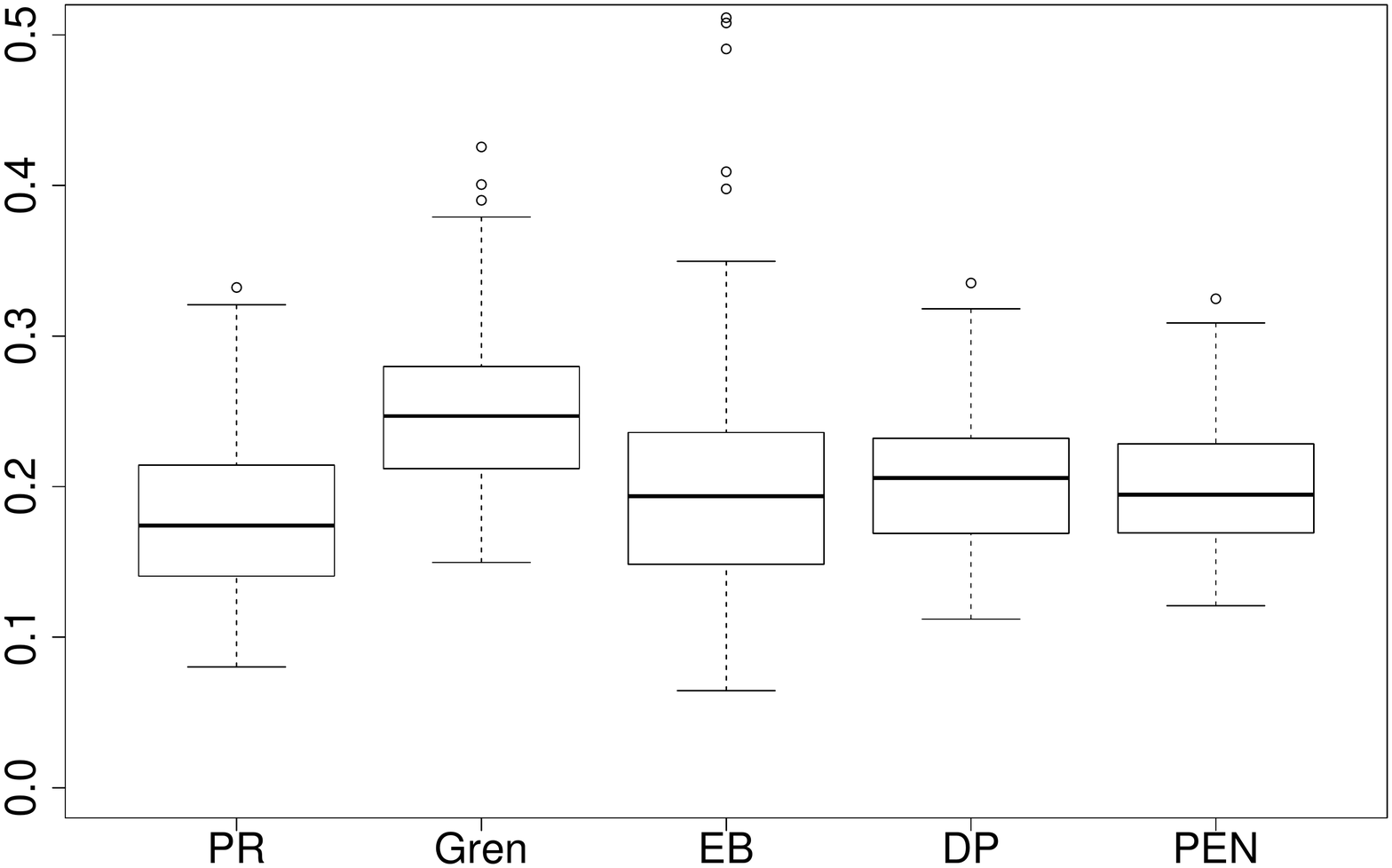}} 
\subfigure[$L_1$ distance, $n = 200$]
{\includegraphics[width = 0.32\linewidth]{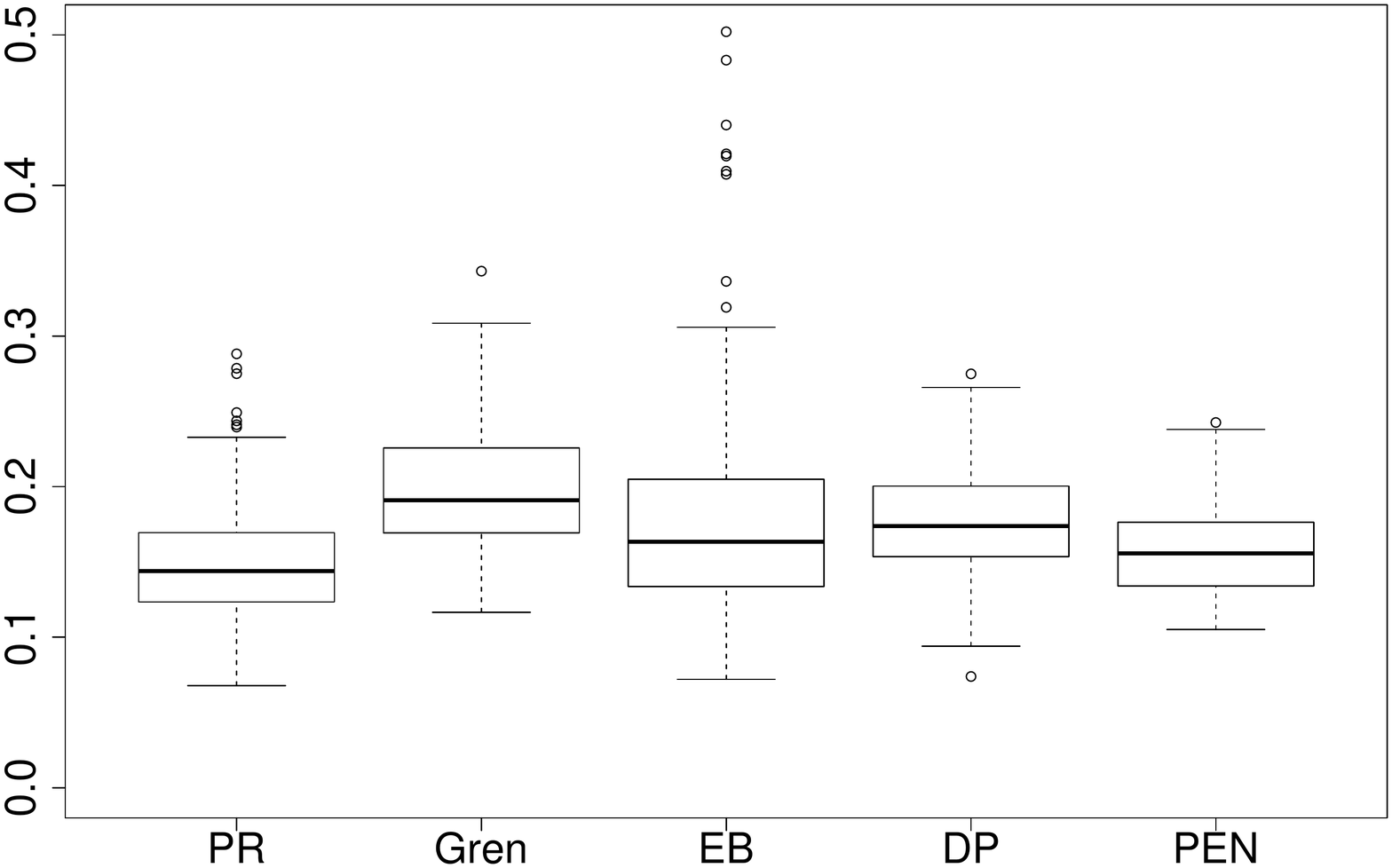}} 
\subfigure[$\hat m(0)/m^\star(0)$, $n=50$]
{\includegraphics[width = 0.32\linewidth]{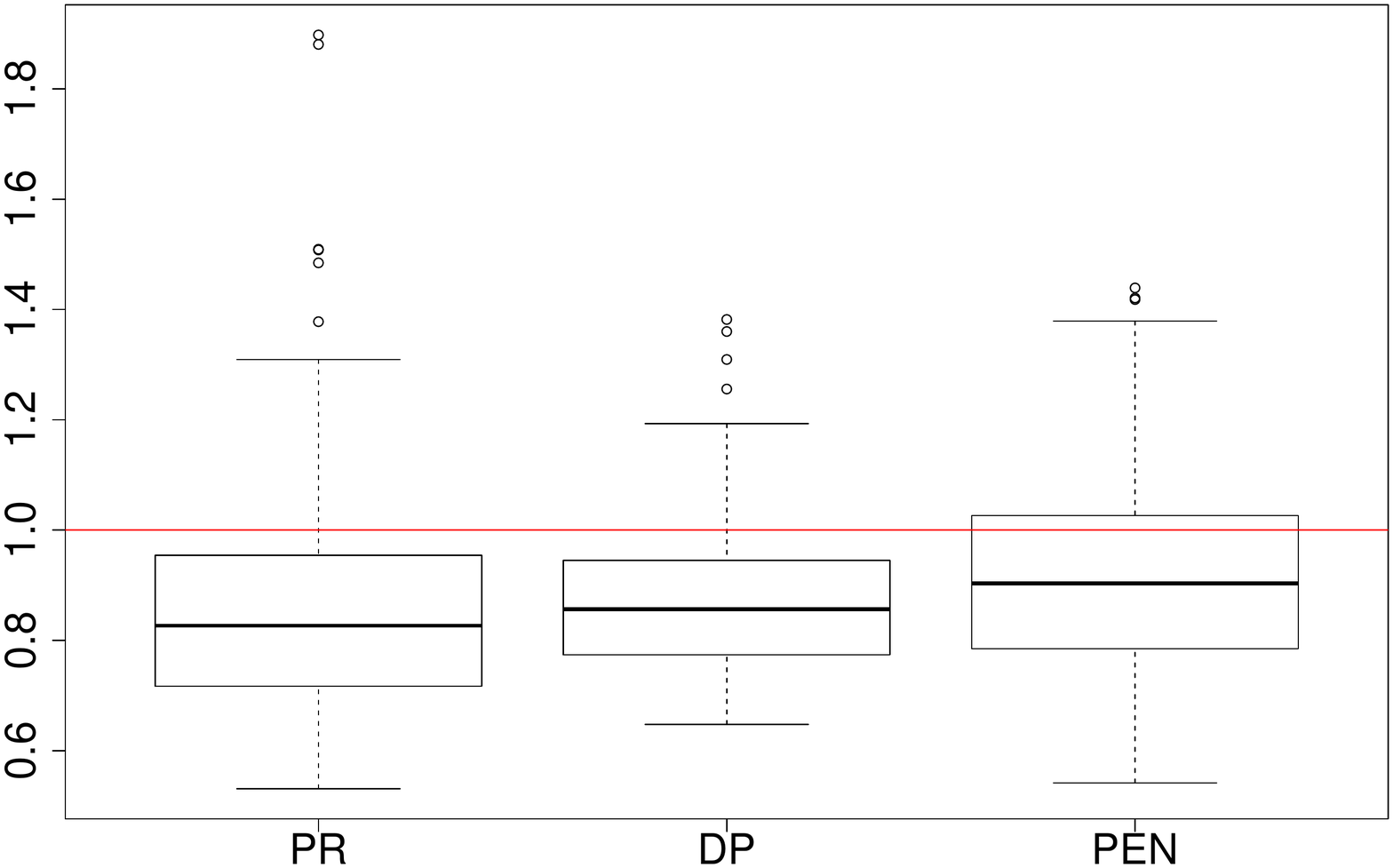}}
\subfigure[$\hat m(0)/m^\star(0)$, $n=100$]
{\includegraphics[width = 0.32\linewidth]{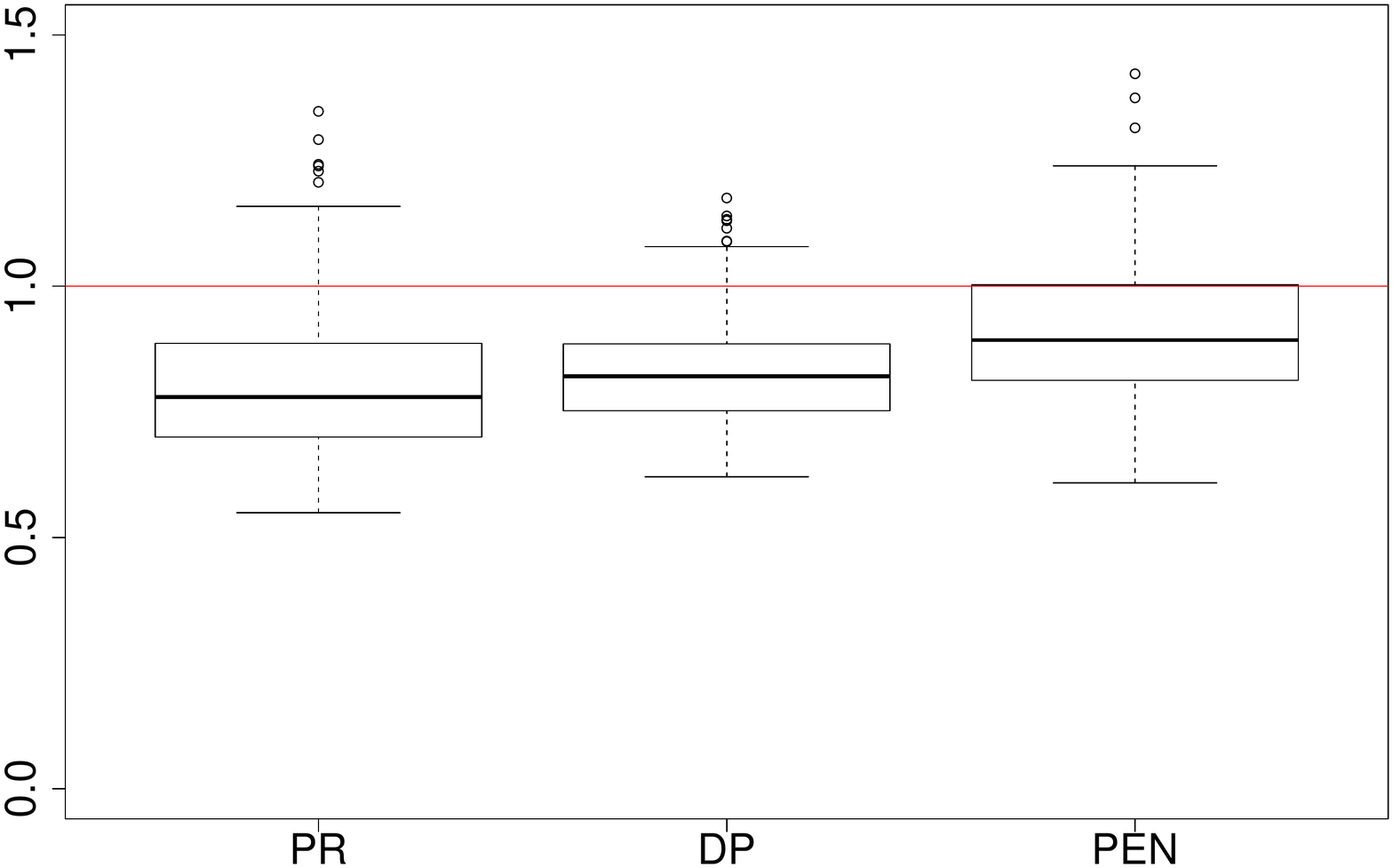}}
\subfigure[$\hat m(0)/m^\star(0)$, $n=200$]
{\includegraphics[width = 0.32\linewidth]{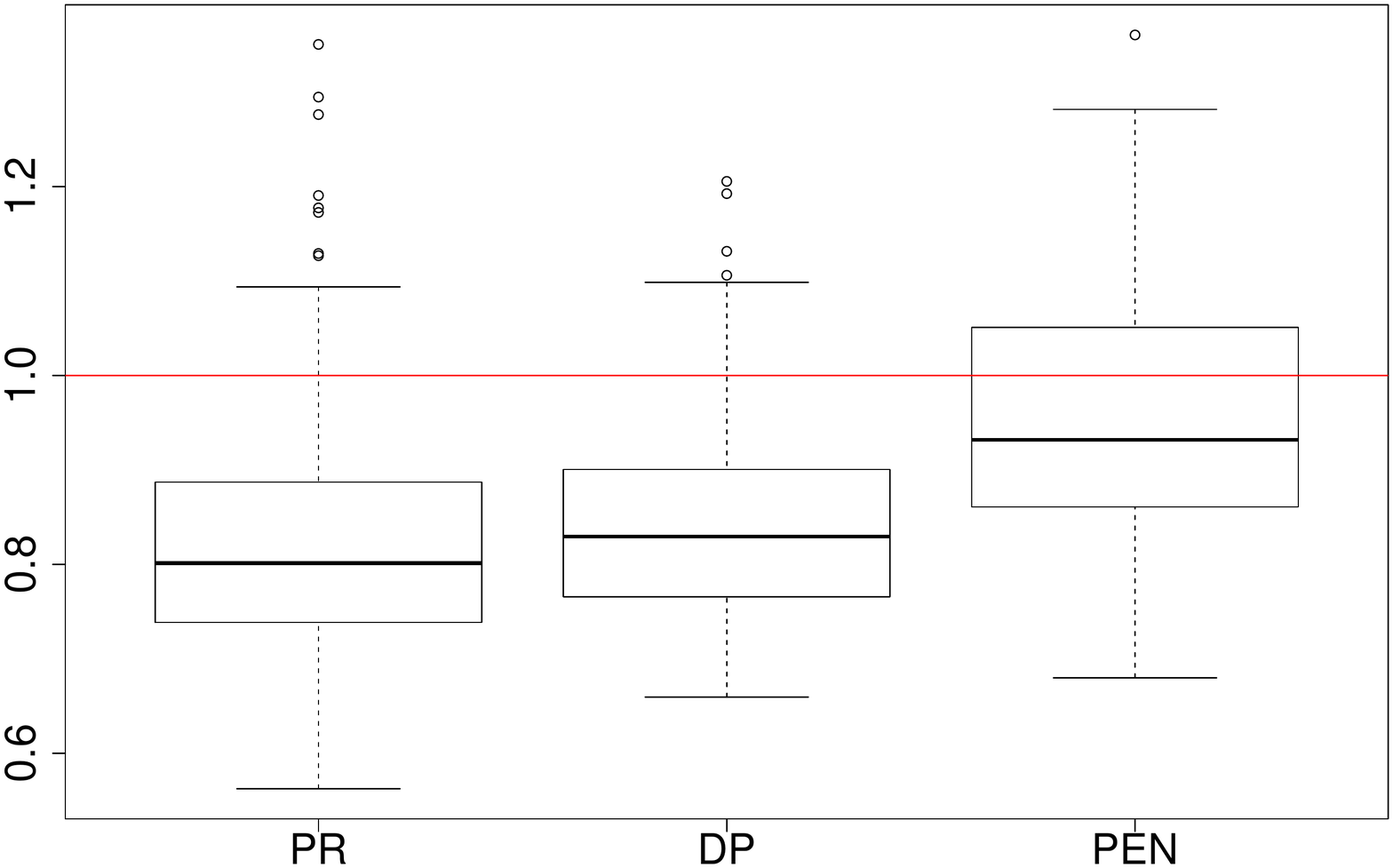}}
\caption{True monotone density is standard exponential}
\label{fig:boxplot_exp}
\end{figure}

\begin{figure}[t]
\centering
\subfigure[$L_1$ distance, $n = 50$]
{\includegraphics[width = 0.32\linewidth]{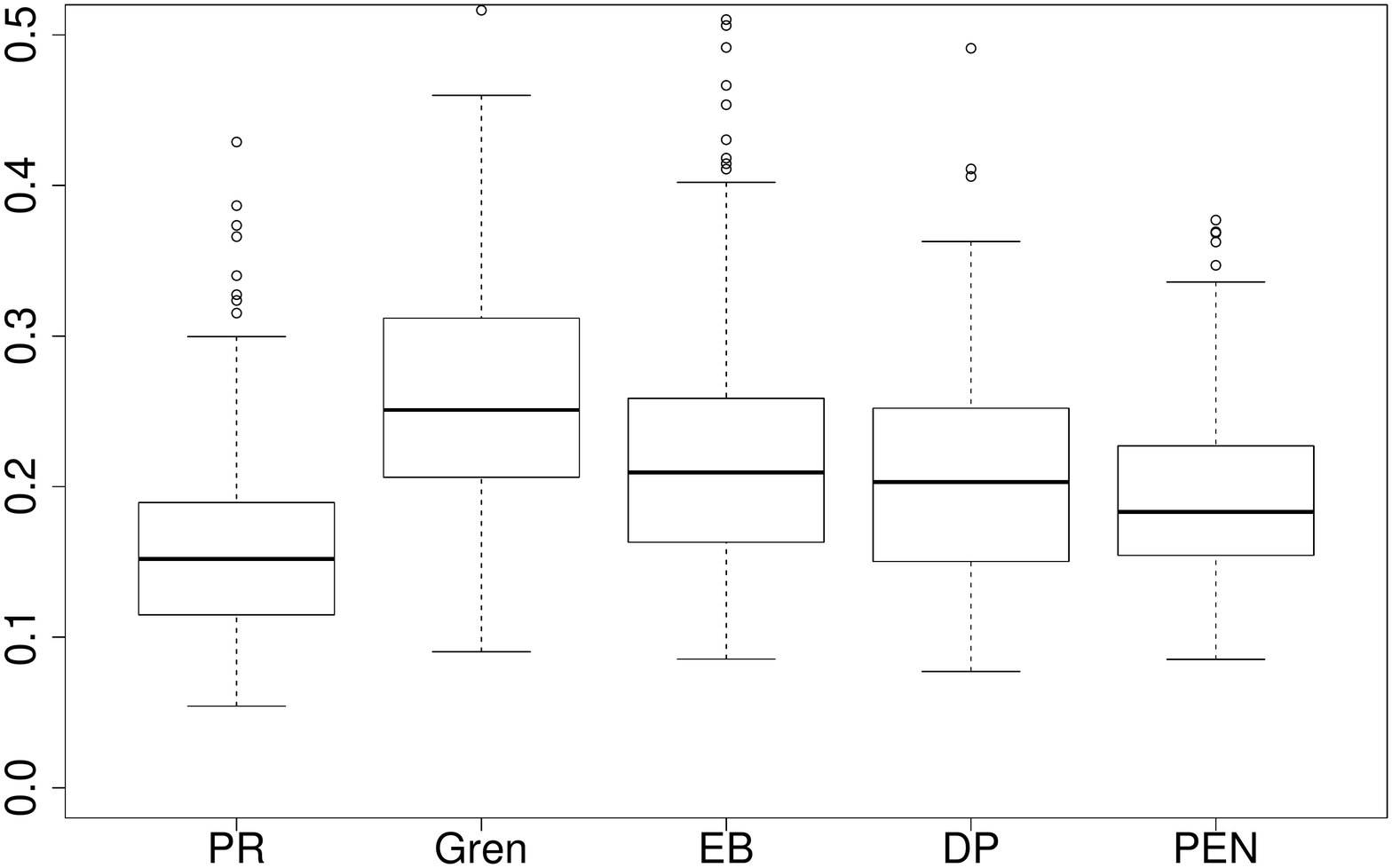}} 
\subfigure[$L_1$ distance, $n = 100$]
{\includegraphics[width = 0.32\linewidth]{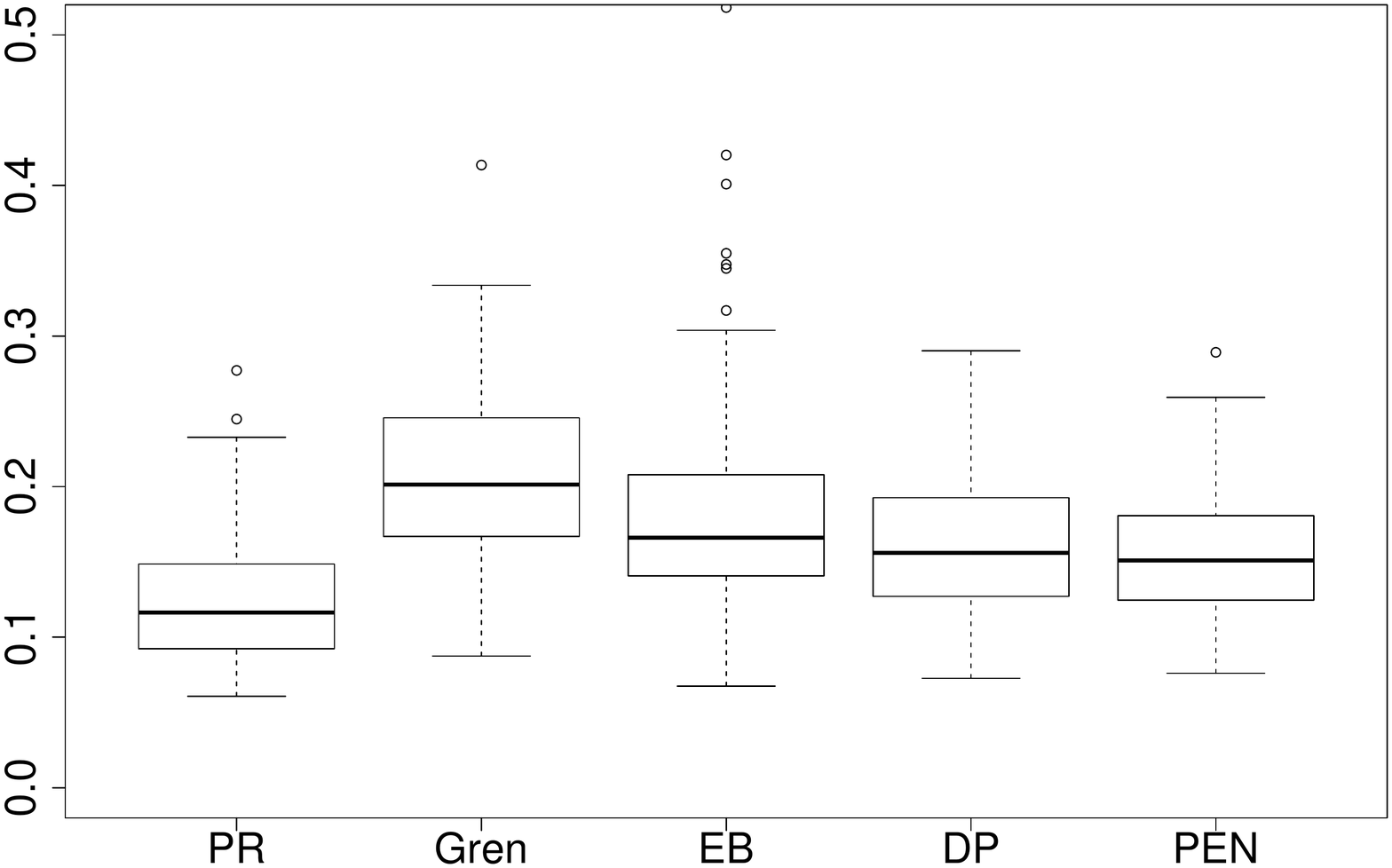}} 
\subfigure[$L_1$ distance, $n = 200$]
{\includegraphics[width = 0.32\linewidth]{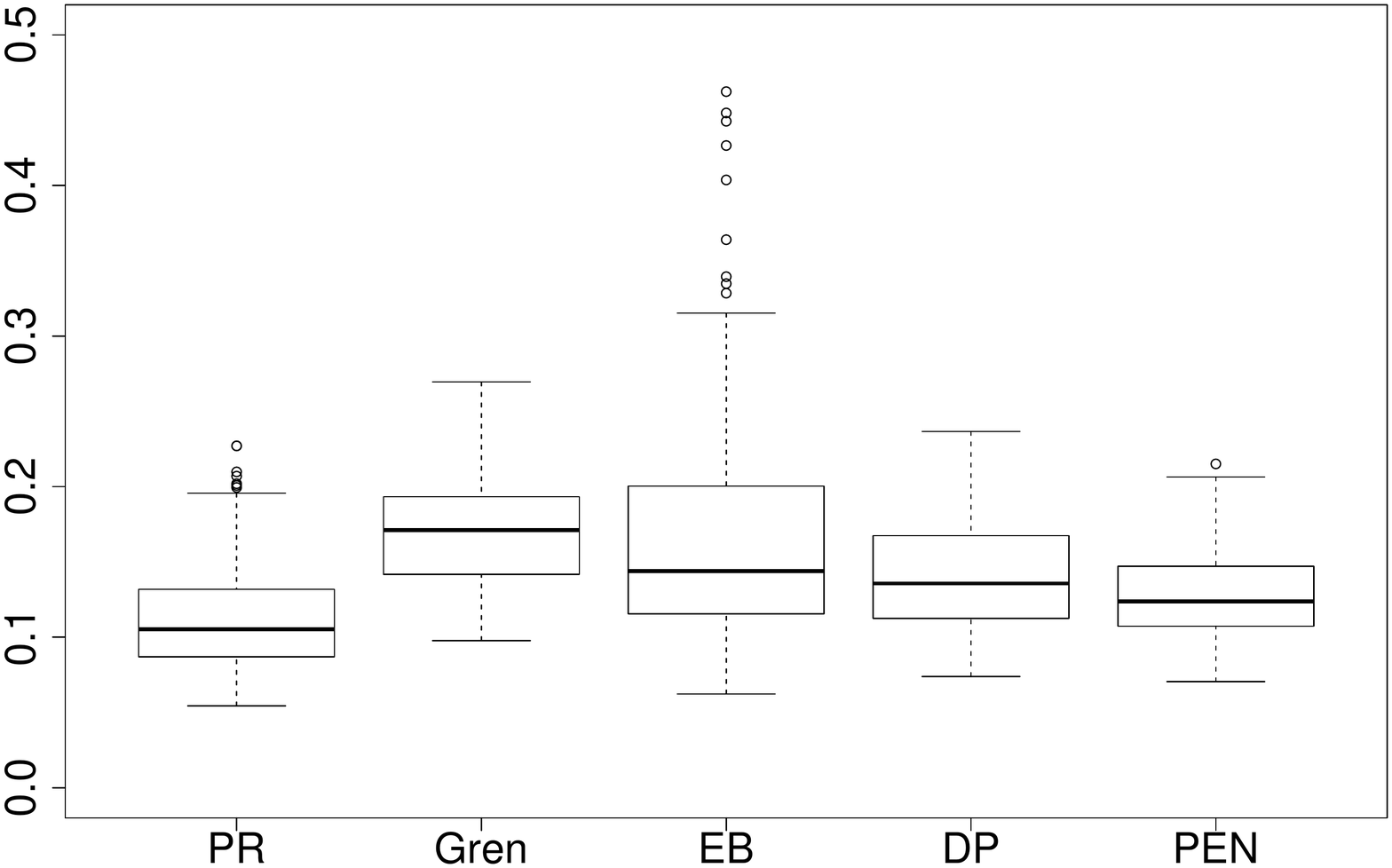}} 
\subfigure[$\hat m(0)/m^\star(0)$, $n=50$]
{\includegraphics[width = 0.32\linewidth]{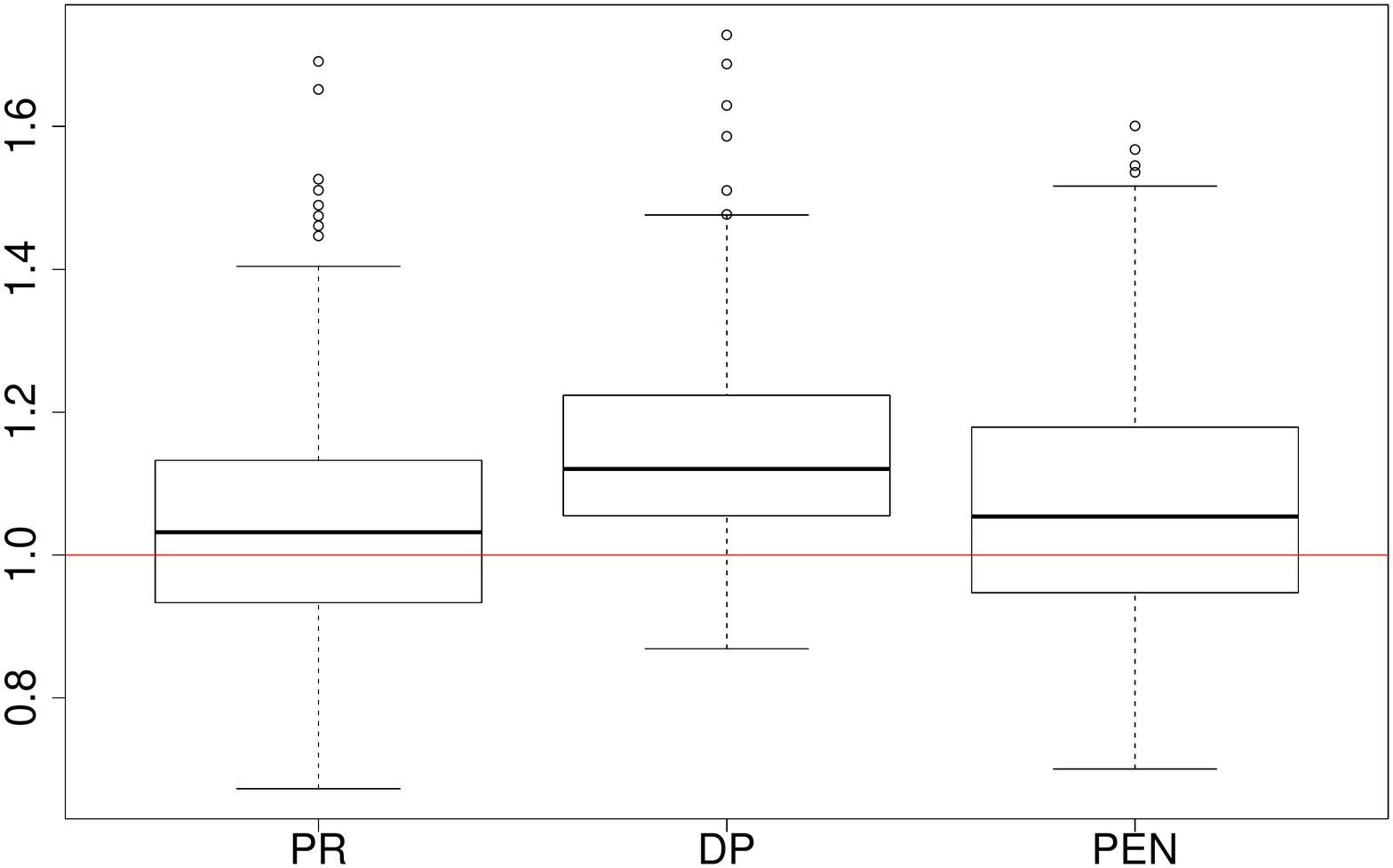}}
\subfigure[$\hat m(0)/m^\star(0)$, $n=100$]
{\includegraphics[width = 0.32\linewidth]{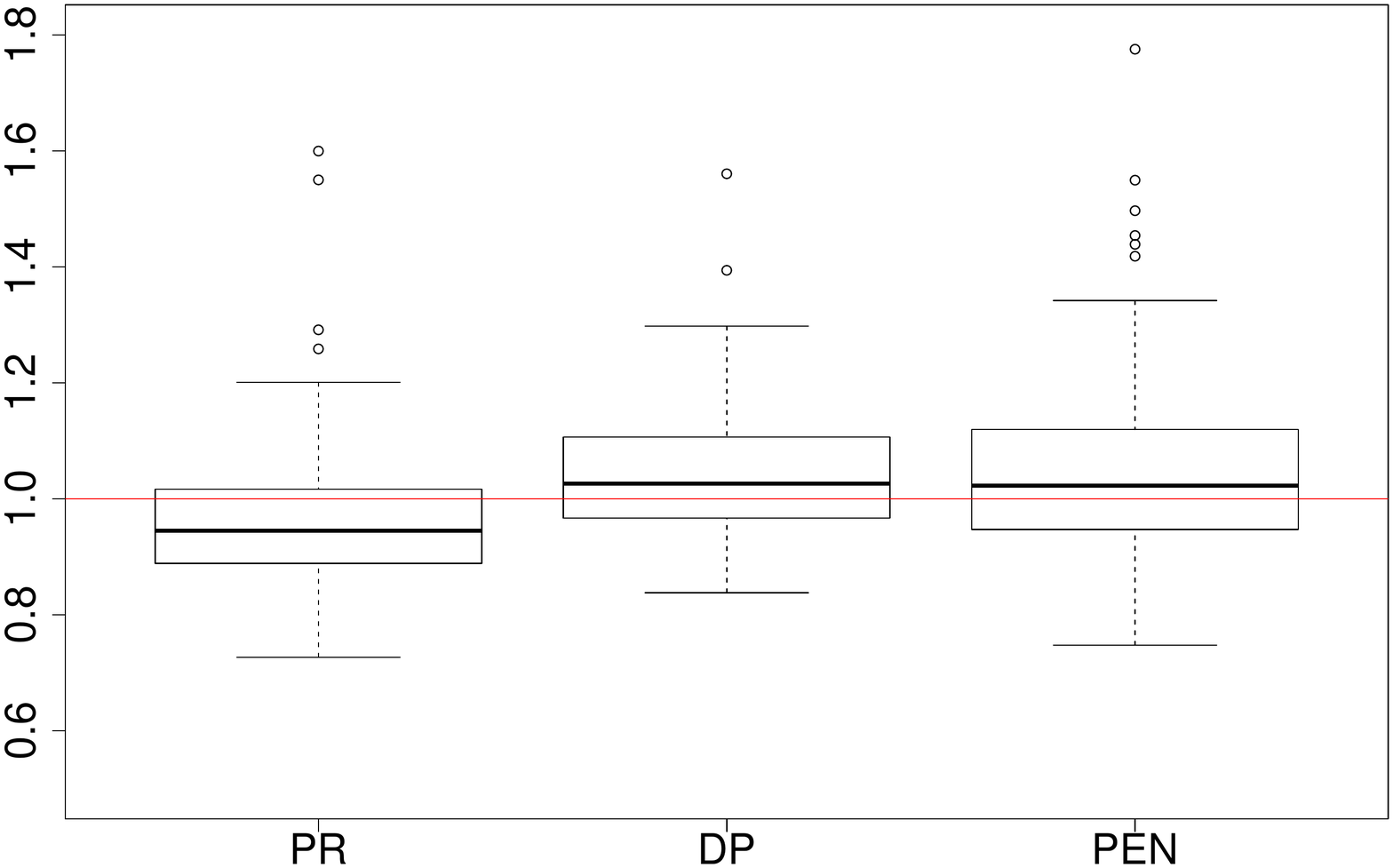}}
\subfigure[$\hat m(0)/m^\star(0)$, $n=200$]
{\includegraphics[width = 0.32\linewidth]{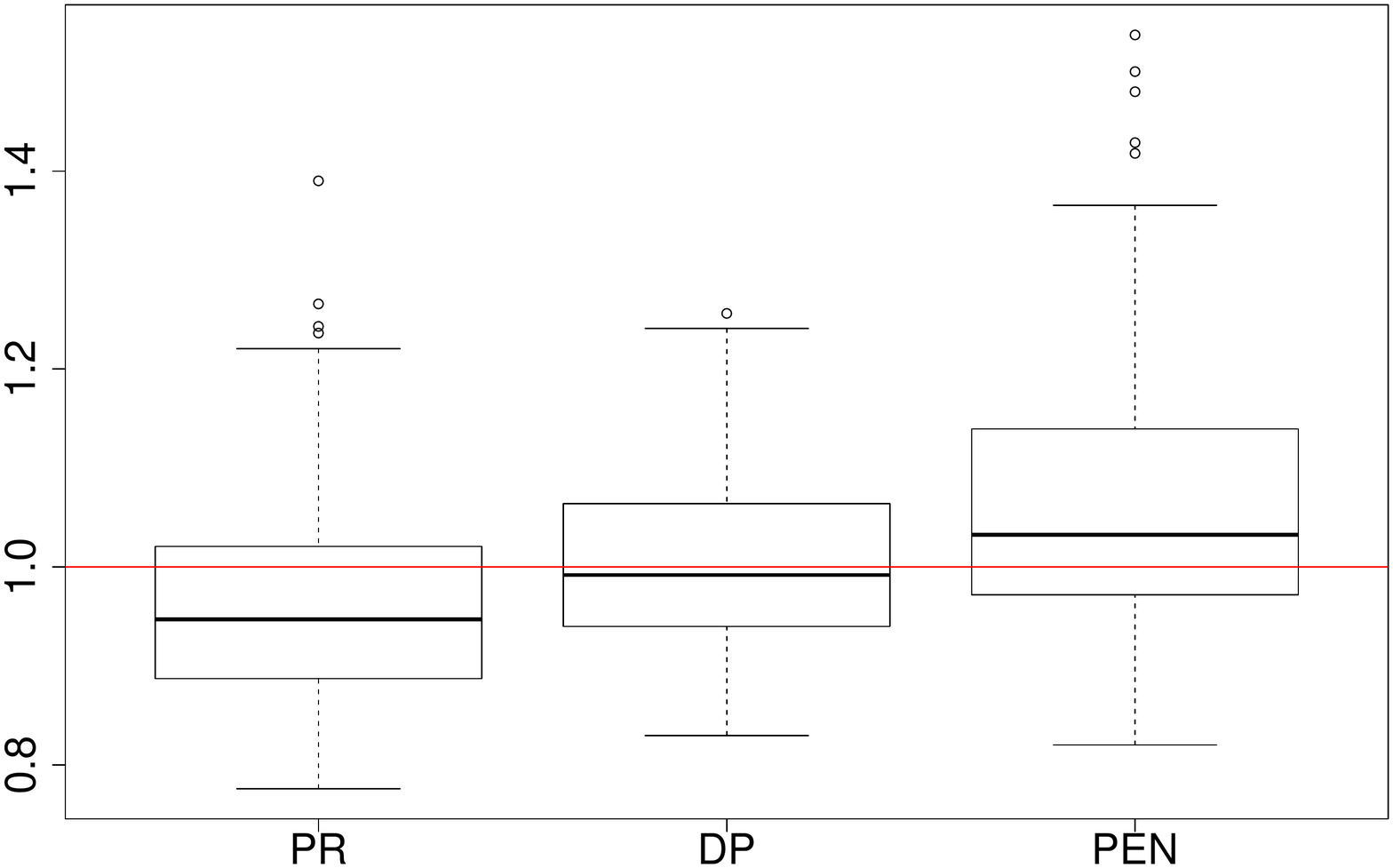}}
\caption{True monotone density is half standard Normal}
\label{fig:boxplot_norm}
\end{figure}

\section{Conclusion}
\label{S:discuss}

Estimation of mixing distributions in mixture models is a challenging problem, one for which there are very few satisfactory methods available.  To our knowledge, the PR algorithm is the one general method available that is both fast and capable of nonparametrically estimating a mixing distribution having a density with respect to any user-specified dominating measure.  Despite the simple and fast implementation of the PR algorithm, and the strong empirical performance observed in numerous applications, its theoretical analysis and justification is non-trivial because of the recursive structure.  Previous work has established consistency of the PR estimates under relatively strong conditions.  Most concerning is that there are known examples, such as monotone density estimation using uniform mixtures, for which the sufficient conditions in previous work do not hold.  The main focus of the present paper was to weaken those overly-strong conditions in order to broaden the range of problems in which PR can be applied.  In particular, the new sufficient conditions can be checked for mixtures of uniform kernels, which puts PR in a position to solve the non-trivial problem of monotone density estimation on $[0,\infty)$.  

There are a number of possible extensions and/or open problems that could be considered.  First, from a practical or methodological point of view, there is a natural extension of the motivating monotone density estimation application.  That is, what can be done if the location of the mode itself is unknown?  This is a non-trivial problem and has been investigated by a number of researchers, including \citet{liughosh2020}.  In the PR framework, the natural approach would be to treat the mode as an unknown, non-mixing parameter contained in the kernel, and apply the PR marginal likelihood strategy in \citet{martintokdar2011} to estimate both the mode and the mode-specific mixing distribution.  How this proposal compares to existing methods remains to be investigated.  

Second, from a theoretical point of view, it is undesirable to work with a fixed and compact mixing distribution support $\UU$.  A natural extension would be to introduce a type of sieve, to allow the support to depend on the sample size, i.e., $\UU = \UU_n$.  The use of a $n$-dependent support $\UU_n$, however, is difficult and awkward in the context of PR.  First, unlike usual likelihood-based methods that assume all the data to be available at once, PR is technically meant to be used for recursive estimation with online data.  In that case, having a sample size dependent support is unnatural since the sample size is not set in advance.  But even if we ignore PR's recursive structure and treat it as being applied to batch data, the analysis is based on martingales that do implicitly treat the data points one by one in a sequence, so having any $n$-specific components in the algorithm itself is awkward.  Beyond awkwardness, there is a specific technical obstacle.  Much of the analysis depends on properties of the functional $T$ defined in \eqref{eq:T}.  This functional depends on $\UU$ and so, if $\UU$ is made to depend on $n$, then we end up with a sequence, $T_n$, of functionals that are applied to the PR sequence of estimates, $P_n$, so new techniques would be needed in order to analyze a sequence of random variables like $T_n(P_{n-1})$.

\section*{Acknowledgments}

This work was supported by the U.S.~National Science Foundation, grant DMS--1737929.

\appendix 

\section{Proofs}
\label{S:proofs}

\subsection{Proof of Theorem~\ref{thm:mixture.KL}}
\label{SS:proof.thm1}

We start by reviewing some details from the analysis in \citet{martintokdar2009}.  From the recursive form of the PR estimate of the mixing distribution, and the linearity of the mixture model, clearly a similar recursive form holds for the mixture. That is, 
\[ m_n(x) = (1 - w_n) \, m_{n-1}(x) + w_n \, h_{n, X_n}(x) \]
where,
\[ h_{n,y}(x) = \frac{\int k(x \mid u) \, k(y \mid u) \, P_{n-1}(du)}{m_{n-1}(y)}, \quad x,y \in \XX. \]
For later, define the function $H_{n,y}(x)$ as
\[ H_{n,y}(x) = \frac{h_{n,y}(x)}{m_{n-1}(x)} - 1 , \quad x,y \in \XX. \]
By Taylor's theorem, we can write 
\[ \log(1 + x) = x - x^2 R(x), \quad x > -1, \]
where the remainder term $R$ satisfies $0 \leq R(x) \leq \max \{ 1, (1+x)^{-2} \}$.  This remainder bound will be important later.  

Let $K_n = K(m^\star, m_n)$.  Then from that recursive form of the mixture density updates above, and this Taylor approximation, it can be shown that 
\[ K_n = K_{n-1} - w_n \int H_{n, X_n}(x) \, m^\star(x) \, dx + w_n^2 \int H_{n, X_n}^{2}(x) \, R(w_n H_{n, X_n}(x)) \, m^\star(x) \, dx. \]
Next, let $\A_r$ denote the $\sigma$-algebra generated by data $X_1,\ldots,X_r$, for $r \geq 1$.  Now take conditional expectation of the above display, given $\A_{n-1}$, to get 
\begin{equation}
\label{martingale}
\E(K_n \mid \A_{n-1}) = K_{n-1} - w_n T(P_{n-1}) + w_n^2 E(Z_n \mid \A_{n-1}),
\end{equation}
where,
\begin{align*}
T(\Phi) & = \int_\UU \left \{\int_\XX \frac{m^{\star}(x)}{m_{\Phi}(x)} k(x \mid u) \, dx \right \}^{2} \, \Phi(du) - 1 \\
Z_n & =  \int_\XX H_{n, X_n}^{2}(x) \, R(w_n H_{n, X_n}(x)) \, m^\star(x) \, dx. 
\end{align*}
If we let $K_n^\star = K_n - K(m^\star, m^\dagger)$, then the same relationship as in \eqref{martingale} holds, i.e., 
\begin{equation}
\label{eq:martingale2}
\E(K_n^\star \mid \A_{n-1}) = K_{n-1}^\star - w_n T(P_{n-1}) + w_n^2 \E(Z_n \mid \A_{n-1}).
\end{equation}
Surprisingly, this form is quite convenient---it is an {\em almost supermartingale} like those studied by \citet{robbinssiegmund}.  Below we restate (a simple version of) Robbins and Siegmund's main theorem for the reader's convenience.  

\begin{rsthm}
Consider non-negative random variables $(M_n, \zeta_n, \xi_n)$, where $(M_n)$ is adapted to a filtration $(\mathcal{A}_n)$.  If
\begin{equation}
\label{eq:rs}
   \E(M_n \mid \mathcal{A}_{n-1}) \leq M_{n-1} - \zeta_{n-1} + \xi_{n-1}.
\end{equation}
and $\sum_n \xi_n < \infty$, almost surely, then $M_n$ converges and $\sum_n \zeta_n < \infty$ almost surely.
\end{rsthm}

The equation in \eqref{eq:martingale2} satisfies the criteria in \eqref{eq:rs}, where $\zeta_{n-1} = w_n T(P_{n-1})$ and $\xi_{n-1} = w_n^2 \E(Z_n \mid \A_{n-1})$.  We need to check that $\sum_n w_n^2 \E(Z_n \mid \A_{n-1})$ is finite almost surely, which amounts to getting a suitable upper bound on $Z_n$ and its conditional expectation.  Here is where our analysis starts to differ from that in \citet{martintokdar2009}.  

The most complicated part of the definition of $Z_n$ is its dependence on the Taylor approximation remainder described above.  Recalling that upper bound, we have 
\[ R(w_n H_{n,X_n}(x)) \leq \max[ 1, \{1 + w_n H_{n, X_n}(x)\}^{-2}]. \]
But since $h_{n,X_n}$ and $m_{n-1}$ are density functions, their ratio is non-negative, so 
\[ w_n H_{n,X_n}(x) = w_n \Bigl( \frac{h_{n,X_n}(x)}{m_{n-1}(x)} - 1 \Bigr) \geq -w_n > -w_1. \]
Therefore, $R(w_n H_{n,X_n}(x)) \leq \max\{ 1, (1 - w_1)^{-2}$, a constant, so  
\[ Z_n \lesssim \int H_{n,X_n}^2(x) \, m^\star(x) \,dx \leq 1 + \int \Bigl( \frac{h_{n,X_n}(x)}{m_{n-1}(x)} \Bigr)^2 \, m^\star(x) \, dx. \]
Since we only need to get an upper bound up to a multiplicative constant, we will ignore that constant lumped inside of ``$\lesssim$'' in what follows; we will also ignore the leading ``$1+$'' since the bound will ultimately get multiplies by $w_n^2$, which itself is summable by assumption.  From this bound, plug in the definition of $h_{n,X_n}$ to get 
\begin{align*}
Z_n & \leq \int \Bigl\{ \frac{\int k(x \mid u) \, k(X_n \mid u) \, P_{n-1}(du)}{m_{n-1}(x) \, m_{n-1}(X_n)} \Bigr\}^2 \, m^\star(x) \, dx \\
& \leq \int \frac{\int k^2(x \mid u) \, k^2(X_n \mid u) \, P_{n-1}(du)}{m_{n-1}^2(x) \, m_{n-1}^2(X_n)} \, m^\star(x) \,dx,
\end{align*}
where the second inequality is by Cauchy--Schwartz.  Next, we focus on one of the terms in the denominator, say, $m_{n-1}(x)$.  From that recursive form for the mixture density updates, we immediately see that 
\[ m_{n-1}(x) \geq (1-w_{n-1}) \, m_{n-2}(x) \geq \cdots \geq m_0(x) \prod_{i=1}^{n-1} (1-w_i), \quad \text{any $x$}. \]
Plug in this lower bound for both terms in the denominator of the bound for $Z_n$ to get 
\[ Z_n \leq \prod_{i=1}^{n-1} (1-w_i)^{-4} \int \frac{\int k^2(x \mid u) \, k^2(X_n \mid u) \, P_{n-1}(du)}{m_0^2(x) \, m_0^2(X_n)} \, m^\star(x) \, dx. \]
Now take conditional expectation with respect to $\A_{n-1}$ and interchange the order of integration (which is allowed since the integrand is non-negative) to get 
\[ \E(Z_n \mid \A_{n-1}) \leq \prod_{i=1}^{n-1} (1-w_i)^{-4} \int \Bigl\{ \int \frac{k^2(x \mid u)}{m_0^2(x)} \, m^\star(x) \, dx \Bigr\}^2 \, P_{n-1}(du). \]
By Condition~\ref{cond:integrable}, we have that the expression inside curly braces above is bounded, uniformly in $u$, by a constant.  Therefore, 
\[ \E(Z_n \mid \A_{n-1}) \lesssim \prod_{i=1}^{n-1} (1-w_i)^{-4}. \]
Next we used the assumed form of the weight sequence, in Condition~\ref{cond:weights}, to bound the above product.  In general, we have 
\[ \log \prod_{i=1}^{n-1} (1-w_i)^{-4} = -4 \sum_{i=1}^{n-1} \log(1-w_i). \]
Using the standard bound, $-\log(1-w) \geq w(1-w)^{-1}$, and the fact that the $w_i$'s are decreasing, we have 
\[ \log \prod_{i=1}^{n-1} (1-w_i)^{-4} = -4 \sum_{i=1}^{n-1} \log(1-w_i) \leq \frac{4}{1-w_1} \sum_{i=1}^{n-1} w_i. \]
According to Condition~\ref{cond:weights}, $w_i = a(i+1)^{-1}$, the summation in the above expression is of the order $\log n$, which implies 
\[ \prod_{i=1}^{n-1} (1-w_i)^{-4} \leq n^{8a/(2-a)}. \]
Putting everything together, we get 
\[ w_n^2 \E(Z_n \mid A_{n-1}) \lesssim n^{-2 + 8a/(2-a)}. \]
Since $a < \frac29$, the exponent is less than $-1$, hence the upper bound is summable almost surely, thus verifying the hypothesis of the Robbins--Siegmund theorem.  Consequently, we can conclude that 
\[ K_n^\star \to K_\infty^\star \quad \text{and} \quad \sum_n w_n T(P_{n-1}) < \infty, \quad \text{almost surely}. \]
It remains to show that the limit, $K_\infty^\star$ is 0 almost surely.  

The key to proving this last claim is an understanding of the properties of the $T$ function.  For a generic mixing distribution $P$, supported on $\UU$, rewrite $T$ as 
\[ T(P) = \int (g_P - 1)^2 \, dP, \]
where
\[ g_P(u) = \int \frac{k \mid u)}{m_P(x)} \, m^\star(x) \, dx. \]
For any bounded and continuous function $h: \UU \to \RR$, it follows from the standard bound $|\int \cdots \, du| \leq \int |\cdots| \, du$ and Cauchy--Schwartz that 
\begin{equation}
\label{eq:cs}
\Bigl| \int (g_P-1) \, h \, dP \Bigr|^2 \leq \Bigl\{ \int |g_P-1| \, |h| \, dP \Bigr\}^2 \leq T(P) \int h^2 \, dP. 
\end{equation}
This implies the lower bound 
\[ T(P) \geq \sup_{h: \int h^2 \,dP = 1} \Bigl\{ \int (g_P-1) \, h \, dP \Bigr\}^2, \]
where the supremum is over all bounded and continuous functions $h$ with $\int h^2 \, dP = 1$.  For an alternative look at the integral in the curly braces above, define the operator $\phi$ that maps a probability measure $P$ on $\UU$ to a new probability measure, $\phi(P)$, on $\UU$ according to the formula 
\[ \phi(P)(A) = \int_A g_P(u) \, P(du), \quad A \subseteq \UU, \, \text{measurable}. \]
Then that expression in curly braces is simply 
\[ \int h \, d\phi(P) - \int h \, dP. \]
A consequence of the Robbins--Siegmund theorem is that $\sum_n w_n T(P_{n-1}) < \infty$ almost surely.  Since $w_n$ itself is vanishing too slowly to be summable, it must be that there exists a subsequence $P_{n(t)}$ such that $T(P_{n(t)}) \to 0$ almost surely.  Therefore, 
\[ \sup_{h: \int h^2 \, dP_{n(t)} = 1} \Bigl\{ \int h \, d\phi(P_{n(t)}) - \int h \, dP_{n(t)} \Bigr\}^2 \to 0, \quad \text{almost surely}. \]
Since the original sequence $P_n$ is tight, there is a sub-subsequence $P_{n(t_s)}$ with a weak limit, and the above result implies that the limit is a fixed point of $\phi$.  However, the only fixed points of this mapping are Kullback--Leibler minimizers, say, $P^\dagger$; see, for example, Lemma~3.4 in \citet{shyamalkumar1996}.  This implies $K_{n(t_s)}^\star$ is vanishing almost surely.  However, by the Robbins--Siegmund theorem, we have that the original sequence $K_n^\star$ converges almost surely to some $K_\infty^\star$.  But if the original sequence has a limit and the limit is 0 on a subsequence, then it must be that $K_\infty^\star = 0$ almost surely.  Putting everything together, we have shown that $K_n^\star = K(m^\star, m_n) - K(m^\star, m^\dagger) \to 0$ almost surely, which implies $K(m^\star, m_n) \to K(m^\star, m^\dagger)$, and completes the proof.

\subsection{Proof of Lemma~\ref{lem:KL.min1}}
\label{SS:proof.lemma1}

The proof proceeds in two steps.  First we express the modified target $m^{\star L}$ as a uniform mixture and identify the corresponding mixing distribution, denoted by $P^{\star L}$.  Then we solve the optimization problem that consists of identifying the mixing distribution, $P^\dagger = P^{\dagger \ell, L}$, supported on $\UU = [\ell, L]$, that minimizes $P \mapsto K(m^{\star L}, m_P)$.  

First, recall the definition of $m^{\star L}$, 
\[ m^{\star L}(x) = \frac{m^\star(x) \, 1_{[0,L]}(x)}{M^\star(L)}, \quad x \in [0,\infty), \]
where $M^\star$ is the distribution function corresponding to the density $m^\star$.  By direct calculation, for the denominator we have 
\[ M^\star(L) = P^\star([0,L]) + L m^\star(L). \]
The numerator can also be rewritten as 
\[ m^\star(x) \, 1_{[0,L]}(x) = \int_0^L \unif(x \mid 0, u) \, P^\star(du) +  m^\star(L) \]
After a bit of algebra to simplify the ratio of the sums in the previous two displays, we are able to write $m^{\star L}$ as a mixture 
\begin{equation}
\label{MstarL}
m^{\star L}(x) = \int \unif(x \mid 0, u) \, P^{\star L}(du), 
\end{equation}
where 
\begin{equation}
\label{eq:PstarL}
  P^{\star L} =  \pi \, \Ptilde^{\star L} + (1-\pi) \, \delta_{\{L\}},
\end{equation}
with $\pi$ and $\Ptilde^{\star L}$ defined as 
\[ \pi = \frac{P^\star([0,L])}{P^\star([0,L]) + L m^\star(L)} \quad \text{and} \quad \Ptilde^{\star L}(du) = \frac{P^{\star}(du)1_{[0,L]}(u)}{P^{\star}([0,L])}. \]
That is, $m^{\star L}$ is a uniform mixture, where the mixing distribution $P^{\star L}$ is not just $P^\star$ restricted and renormalized to $[0,L]$, but a mixture of that and a point mass at $L$.  

For step 2, we want to find the minimizer of $P \mapsto \kappa(P) := K(m^{\star L}, m_{P})$, over all mixing distributions supported on $\UU = [\ell, L]$, where $m^{\star L}$ has the mixture form presented above. Using the above notation, the lemma's claim is that the minimizer is 
\[ P^\dagger = \omega \, \delta_{\{\ell\}} + P^{\star L}|_\UU, \]
where $P^{\star L}|_\UU$ is $P^{\star L}$ restricted (but not renormalized) from $[0,L]$ to $\UU = [\ell, L]$, and $\omega = P^{\star L}([0,\ell])$. If we can show that the Gateaux derivative of $\kappa$, evaluated at $P^\dagger$, in the direction of any other distribution $H$ on $\UU$, is vanishing, then we will have proved the claim.  The Gateaux derivative at a generic $P$, in the direction of $H$, is 
\[ \frac{d}{dt} \kappa((1-t)P + tH) \Bigr|_{t=0} = \int_0^L \Bigl\{ 1 - \frac{m_H(x)}{m_P(x)} \Bigr\} \, m^{\star L}(x) \, dx. \]
Let $m^\dagger = m_{P^\dagger}$, which has the form 
\[ m^\dagger(x) = \omega \, \unif(x \mid 0, \ell) + \int_\ell^L \unif(x \mid 0, u) \, P^{\star L}(du). \]
Then the goal is to show that 
\[ \int_0^L \Bigl\{ 1 - \frac{m_H(x)}{m^\dagger(x)} \Bigr\} \, m^{\star L}(x) \, dx = 0 \quad \text{for all $H$ supported on $\UU$}, \]
or, equivalently, to show that 
\begin{equation}
\label{eq:gat1}
1 - \int_0^\ell \frac{m_H(x)}{m^\dagger(x)} \, m^{\star L}(x) \, dx - \int_\ell^L \frac{m_H(x)}{m^\dagger(x)} \, m^{\star L}(x) \, dx = 0    
\end{equation}
On the interval $x \in (\ell, L]$, it is clear that $m^\dagger(x) = m^{\star L}(x)$, so 
\begin{equation}
\label{eq:gat2}
\int_\ell^L \frac{m_H(x)}{m^\dagger(x)} \, m^{\star L}(x) \,dx = \int_\ell^L m_H(x) \, dx. 
\end{equation}
Next, since both $P^\dagger$ and $H$ are supported on $\UU=[\ell,L]$, the two mixture densities $m^\dagger$ and $m_H$ are constant on the interval $x \in [0,\ell]$.  This implies 
\[ \int_0^\ell \frac{m_H(x)}{m^\dagger(x)} \, m^{\star L}(x) \, dx - \int_0^\ell m_H(x) \, dx = \frac{m_H(0)}{m^\dagger(0)} \int_0^\ell \{ m^{\star L}(x) - m^\dagger(x)\} \, dx. \]
We claim that the integral on the right-hand side is 0.  To see this, first integrate $m^\dagger$:
\begin{align*}
\int_0^\ell m^\dagger(x) \, dx & = \omega + \int_0^\ell \int_\ell^L \unif(x \mid 0, u) \, P^{\star L}(du) \, dx \\
& = \omega + \ell m^{\star L}(\ell) \\ 
& = \frac{P^\star([0,L]) + \ell m^\star(\ell)}{M^\star(L)}.
\end{align*}
Similarly, integrate $m^{\star L}$: 
\begin{align*}
\int_0^\ell m^{\star L} (x) dx & = \frac{1}{M^\star(L)} \int_0^\ell \Bigl\{ \int_0^L \unif(x \mid 0, u) \, P^{\star}(du) + m^{\star} (L) \Bigr\} dx \\
& = \frac{1}{M^\star(L)} \Bigl\{ \int_\ell^L (\ell /u) P^\star (du) + \int_0^\ell P^\star(du) + \ell m^\star(L) \Bigr\} \\
& = \frac{1}{M^\star(L)} \bigl\{ \ell m^\star(\ell) - \ell m^\star(L) + P^\star([0,\ell]) + \ell m^\star(L) \bigr\} \\
& = \frac{P^\star([0,\ell]) + \ell m^\star(\ell)}{M^\star(L)}. 
\end{align*}
Clearly the two integrals above are the same, which implies that 
\[ \int_0^\ell \{ m^{\star L}(x) - m^\dagger(x)\} \, dx = 0, \]
and, consequently, that 
\begin{equation}
\label{eq:gat3}
\int_0^\ell \frac{m_H(x)}{m^\dagger(x)} \, m^{\star L}(x) \, dx = \int_0^\ell m_H(x) \, dx. 
\end{equation}
Plugging the relations \eqref{eq:gat2} and \eqref{eq:gat3} into the left-hand side of \eqref{eq:gat1} proves the claim, i.e., that the Gateaux derivative of $\kappa$ at $P^\dagger$ vanishes in all directions $H$, which implies that $P^\dagger$ is the minimizer of the Kullback--Leibler divergence.

\subsection{Proof of Theorem~\ref{thm:monotone}}
\label{SS:proof.thm2}

To prove $K(m^{\star L}, m_n) \to K(m^{\star L}, m^\dagger)$, we apply Theorem~\ref{thm:mixture.KL}.  Condition~\ref{cond:weights} is in the user's control and, hence, is easy to satisfy.  Condition~\ref{cond:support} requires the support of the mixing distribution to be compact, which is clearly satisfied by $\UU = [\ell, L]$.  Condition~\ref{cond:integrable} is the only non-trivial condition, and it requires
\[ \sup_{u \in [\ell, L]} \int_0^L \Bigl\{ \frac{\unif(x \mid 0, u)}{m_0(x)} \Bigr\}^2 \, m^{\star L}(x) \, dx < \infty, \]
where $m_0$ is the mixture density corresponding to the initial guess, $P_0$, which contains point masses.  The key point is, thanks to the point mass at $L$, 
\[ m_0(x) \geq p_{0,L} \, \unif(x \mid 0, L) = p_{0,L} \, L^{-1}, \quad x \in [0, L]. \]
Since the denominator above is uniformly bounded away from 0, and, similarly, the numerator is uniformly bounded by $\ell^{-1}$, Condition~\ref{cond:integrable} clearly holds. 

Next, the claim about convergence of $m_n$ to $m^\dagger$ in total variation follows immediately from Corollary~\ref{cor:mixture.L1} and the fact that $m^{\star L}$ is bounded away from 0.  Finally, for the claim about weak convergence of $P_n$ to $P^\dagger$, we apply Corollary~\ref{cor:mixing}.  We have already stated that $m^\dagger / m^{\star L} \in L_\infty$ since $m^{\star L}$ is bounded away from 0.  So all that remains is to check that the uniform kernel satisfies the abstract condition \eqref{eq:kernel}, which we do next.  

Imagine a generic sequence of mixing distributions $Q_t$ supported on $\UU=[\ell,L]$ and assume they converge weakly to $Q_\infty$.  The condition \eqref{eq:kernel} concerns the behavior of the mixture density $m_{Q_t}(x)$.  Note that the uniform kernel is not a continuous function in $u$ for a given $x$, but it is upper-semicontinuous.  Recall that the mixture densities are constant for $x \in [0,\ell]$. This means that the value of the mixture density on a set of positive measure is determined by its value at $x=\ell$, so some care will be needed below; in particular, we'll have to deal with the cases $x \in [0,\ell]$ and $x \in (\ell, L]$ separately.  

Start with the case $x \in (\ell, L]$.  The kernel $u \mapsto \unif(x \mid 0, u)$ is bounded and continuous except for the jump discontinuity at $u=x$.  It is possible that the limit $Q_\infty$ of the sequence $Q_t)$ of mixing distributions puts positive mass at $u=x$, i.e., that $x$ is a discontinuity point of $Q_\infty$.  In such cases, $m_{Q_t}(x)$ may not converge or, even if it does converge, the limit may not equal $m_{Q_\infty}(x)$.  However, $Q_\infty$'s set of discontinuity points has Lebesgue measure 0.  For any $x \in (\ell, L]$ that is not a discontinuity point of $Q_\infty$, the kernel is effectively bounded and continuous, so $Q_t \to Q_\infty$ weakly implies $m_{Q_t}(x) \to m_{Q_\infty}(x)$.  This verifies \eqref{eq:kernel} for the range $x \in (\ell,L]$.  

For the case $x \in [0,\ell]$, again, we know that the mixture density is constant in $x$.  Therefore, if there is an issue with convergence of the mixture density at $x=\ell$, then that implies an issue on a set of positive Lebesgue measure, hence \eqref{eq:kernel} fails.  However, while the kernel is only upper-semicontinuous in general, $u \mapsto \unif(\ell \mid 0, u)$ is bounded and continuous on the support of the $Q_t$ sequence, so we get $m_{Q_t}(\ell) \to m_{Q_\infty}(\ell)$ automatically from the definition of weak convergence.  This implies the same for all $x \in [0,\ell]$, so \eqref{eq:kernel} holds there too.




\subsection{Proof of Proposition~\ref{prop:bias}}
\label{proof:bias}

By the triangle inequality, we have 
\begin{equation}
\label{eq:bias.bound}
\int |m^{\dagger} - m^\star| \,dx \leq \int |m^{\dagger} - m^{\star L}| \, dx + \int |m^{\star L} - m^\star| \, dx. 
\end{equation}
Now we consider each term in the upper bound \eqref{eq:bias.bound} separately.  Start with the second term, splitting up the range of integration, we immediately get 
\begin{align*}
\int |m^{\star L} - m^\star| \, dx 
&= \int_0^L \Bigl| \frac{m^{\star}}{M^{\star}(L)} -  m^{\star} \Bigr| \, dx + 1 - M^{\star}(L) \\
    &= \frac{1}{M^{\star}(L)} |1 - M^{\star}(L) | \int_0^L m^{\star} \, dx + 1 - M^{\star}(L) \\
    &= 2\{1 - M^{\star}(L)\}. 
\end{align*}
For the first term in \eqref{eq:bias.bound}, we borrow the calculations in the proof of Lemma~\ref{lem:KL.min1} above.  In particular, on the interval $x \in [\ell, L]$, the two densities are the same, but on the interval $x \in [0,\ell)$, the absolute difference between densities is bounded by
\[ |m^{\star L}(x) - m^\dagger(x)| \leq \omega \unif(x \mid 0, \ell) + \int_0^\ell \unif(x \mid 0, u) \, P^{\star L}(du), \quad x \in [0,\ell). \]
Now integrate to get 
\begin{align*}
\int |m^{\dagger} - m^{\star L}| \, dx &= \int_0^\ell |m^{\dagger} - m^{\star L}| \, dx\\
& \leq \omega + P^{\star L}([0,\ell]) \\
& = 2 \cdot \frac{P^\star([0,\ell])}{M^\star(L)}.
\end{align*}
Combining the two bounds proves the claim.

\subsection{Proof of Proposition~\ref{prop:origin}}
\label{proof:origin}

As shown in the proof of Theorem~\ref{thm:monotone}, $m_n (\ell) \to m^\dagger(\ell)$ almost surely with respect to $m^{\star L}$.  Since $m_n(0) = m_n(\ell)$ and $m^\dagger(0) = m^\dagger(\ell)$ by Equation \eqref{eq:monotone}, the proof of the first claim is complete.  To bound the bias, i.e., the difference between the quantity being estimated, $m^\dagger(0)$, and and the true density at the origin, $m^\star(0)$, we proceed as follows.  
\begin{align*}
m^\dagger (0) - m^\star (0) &= a_\ell \ell^{-1} + a_\UU \int_\UU u^{-1} \, P^\star(du) + a_L L^{-1} - \int_0^\infty u^{-1} P^\star (du)\\
&= \Bigl\{ a_\ell \ell^{-1} -  \int_0^\ell u^{-1} P^\star (du) \Bigr\} + \Bigl\{ (a_\UU - 1) \int_\UU u^{-1} \, P^\star(du)  \Bigr\} \\
    & \qquad + \Bigl\{ a_L L^{-1} -  \int_L^\infty u^{-1} P^\star (du) \Bigr\}. 
\end{align*}
Using the definitions of $a_\ell$, $a_\UU$, and $a_L$, the bound $P^\star ([0, \ell]) \lesssim \ell$, and the fact that $\int_\UU u^{-1} \, P^\star(du) = O(1)$ as a function of $(\ell, L)$, it is easy to check that each of the three terms on the right-hand side above can be bounded by $1-M^\star(L)$.  That is, 
\begin{align*}
a_\ell \ell^{-1} -  \int_0^\ell u^{-1} \, P^\star(du) & \lesssim M^\star(L)^{-1} - 1 \lesssim 1-M^\star(L) \\
(a_\UU - 1) \int_\UU u^{-1} \, P^\star(du)  & \lesssim 1 - M^\star(L) \\
a_L L^{-1} -  \int_L^\infty u^{-1} \, P^\star(du) & \lesssim 1 - M^\star(L),
\end{align*}
which completes the proof of the claim. 

\bibliographystyle{apalike}
\bibliography{refs}

\end{document}